\documentclass[12pt]{amsart}
\usepackage{epsfig,amsthm,amsmath,amsfonts,amssymb,rotating,nicefrac}
\usepackage{tikz,pgfplots}
\usepackage{pgfplotstable}
\usepackage{tkz-euclide}
\pgfplotsset{compat=newest}
\usetikzlibrary{intersections, positioning, quotes, angles}
\usetikzlibrary{calc}
\usepackage{caption}
\usepackage{subcaption}
\usepackage[english]{babel}
\usepackage[utf8]{inputenc}
\usepackage[T1]{fontenc}
\usepackage{bbm, bm}
\usepackage{subcaption}

\usepackage{hyperref}
\hypersetup{
    colorlinks,
    citecolor=black,
    filecolor=black,
    linkcolor=black,
    urlcolor=black
}
\definecolor{green}{rgb}{0,0.5,0}

\allowdisplaybreaks[1]

\newtheorem{theorem}{Theorem}[section]
\newtheorem{lemma}[theorem]{Lemma}

\newtheorem{remark}[theorem]{Remark}
\numberwithin{equation}{section}

\renewcommand{\d}{\operatorname{d}\!}

\makeatletter

\newcommand{\Vbfm}{\mathbf{V}}

\newcommand{\on}[1]{\operatorname{#1}}
\newcommand{\drm}{\mathrm{d}}
\newcommand{\nbfm}{\mathbf{n}}
\newcommand{\Zbfm}{\mathbf{Z}}
\newcommand{\xbfm}{\mathbf{x}}
\newcommand{\ybfm}{\mathbf{y}}

\newcommand{\kabfm}{\boldsymbol{\kappa}} 
\newcommand{\xibfm}{\boldsymbol{\xi}}

\newcommand{\ALEoV}[2]{\mathbf{T}(#1,\mathbf{#2})}

\newcommand{\red}[1]{\textcolor{red}{#1}}

\newcommand{\V}{\mathcal{V}}
\newcommand{\K}{\mathcal{K}}

\begin{document}
\title[Solution of a time-dependent shape optimization problem]
{On the numerical solution of a time-dependent shape 
optimization problem for the heat equation}
\author{Rahel Br\"ugger}
\author{Helmut Harbrecht}
\author{Johannes Tausch}
\address{Rahel Br\"ugger, Helmut Harbrecht,
Departement Mathematik und Informatik,
Universit\"at Basel,
Spiegelgasse 1, 4051 Basel, Schweiz.}
\email{\{ra.bruegger,helmut.harbrecht\}@unibas.ch}
\address{Johannes Tausch, 
Department of Mathematics, 
Southern Methodist University, Dallas, TX 75275.}
\email{tausch@smu.edu}
\date{}
\begin{abstract}
  This article is concerned with the solution of a time-dependent
  shape identification problem. Specifically we consider the heat
  equation in a domain, which contains a time-dependent inclusion of
  zero temperature. The objective is to detect this inclusion from the
  given temperature and heat flux at the exterior boundary of the
  domain.  To this end, for a given temperature at the exterior
  boundary, the mismatch of the Neumann data is minimized.  This
  time-dependent shape optimization problem is then solved by a
  gradient-based optimization method. Numerical results are presented
  which validate the present approach.
\end{abstract}

\keywords{Inverse problem, shape optimization, heat equation}
\maketitle
\section{Introduction}
Shape optimization appears in a wide range of problems from
engineering, especially for designing and constructing industrial 
components or in non-destructive testing. Many practical problems 
from engineering amount to partial differential equations for an unknown 
function, which needs to be computed to obtain the quantity of interest. 
Shape optimization is then concerned with the minimization of this quantity 
of interest. While shape optimization in case of \emph{elliptic partial 
differential equations\/} is a well studied topic in literature, see for example 
\cite{Delfour_Zolesio,Soko_Zolesio} and the references therein, not so 
much is known about shape optimization in case of \emph{parabolic
partial differential equations\/}.

Theoretical results for parabolic shape optimization problems with 
\emph{time-independent shapes\/} can be found in \cite{sokolowski1988shape, Soko_Zolesio, El_Yacoubi}, while practical results are found 
for example in \cite{Chapko_et_al,chapko_neumann, HHTauschCat}. 
This is in contrast to the results for parabolic shape optimization 
problems with \emph{time-dependent shapes\/}. Theoretical results 
are for example available in \cite{Dziri_dynamical,Dziri,Moubachir}, 
but to the best of our knowledge, no results about efficient computations 
of such time-dependent shape optimization problems exist.

This article is based on the previous article \cite{HHTauschCat} by
two of the authors, where a parabolic shape optimization problem 
is considered for a time-independent shape. The goal therein was 
to detect a fixed inclusion or void of zero temperature inside a 
three-dimensional solid or liquid body by measurements of the 
temperature and the transient heat flux at the accessible outer 
boundary. Since the underlying shape calculus turned out to be 
rather standard due to the stationarity of the inclusion, the focus 
has been on the development of an efficient solver for the underlying 
heat equation. In contrast, in the present article, we now consider an 
inclusion, which changes its shape during time. Therefore, the shape 
calculus becomes the focus, while the numerical experiments are 
performed in two space dimensions and serve as a proof of concept.

The problem under consideration is reformulated as a shape optimization 
problem by means of a tracking-type functional for the Neumann data.
Therefore, for given temperature at the exterior boundary, the mismatch
of the Neumann data is minimized in a least-squares sense. Since we 
intend to apply a gradient-based optimization algorithm, we compute the 
shape gradient of this functional by means of the adjoint approach, which 
is known to reduce the computational effort. Then, we make a parametric 
ansatz for the inclusion and use a boundary element method to solve
the heat equations for the primal state and the adjoint state. Numerical
results validate that the present approach is feasible, leading to meaningful
reconstructions.

The remainder of the article is organized as follows. In 
Section~\ref{sec.problem_formulation}, we state the problem under 
consideration. Section~\ref{sec.shape_calc_neumann} is dedicated 
to the time-dependent shape calculus of our functional. 
Section~\ref{sec.discretization_shape_opt} shows how we can 
discretize our problem in the case of a void which is star-shaped
for all points of time. In order to solve the heat equation on the current 
domain, Section~\ref{sec.solving_parabolic_BVP} explains how to do 
this by using a boundary element method. Since the method parallels
that of \cite{HHTauschCat}, this section only discusses the changes 
for the moving boundaries considered in this article. In order to illustrate 
the developed techniques, they are applied to the example shown
in Section~\ref{sec.results}. Finally, in Section \ref{sec.conclusion}, 
we give some concluding remarks.

\section{Problem formulation}\label{sec.problem_formulation}
\subsection{Model problem}
Let $D \subset \mathbb{R} ^d$ with $d = 2,3$ be a simply connected, 
spatial domain with boundary $\Gamma^f = \partial D$. Moreover, we 
have a time component, and thus the domain $(0,T)\times D$ forms a 
cylindrical domain, called the space-time cylinder. At every time 
$t\in[0,T]$, a simply connected subdomain $S_t \subset D$ with boundary 
$\Gamma_t = \partial S_t$ lies inside $D$ such that it holds $\on{dist}
(\Gamma^f,\Gamma_t) > 0$ for all $t$. The difference domain is called 
$\Omega_t := D\setminus\overline{S_t}$. Taking into account the time 
again, we thus consider tubes (i.e., non-cylindrical domains), which 
contain a void and are represented as
\[
Q_T= \underset{0 < t< T}{\bigcup} (\{ t \} \times \Omega_t).
\]
The interior boundary of the tube $Q_T$ is called
\[ 
  \Sigma_T = \underset{0 < t < T}{\bigcup} ( \{t \} \times\Gamma_t ) 
\]
and the exterior boundary of the tube is called $\Sigma^f = (0,T) 
\times\Gamma^f$.\footnote{We assume that the exterior boundary 
$\Gamma^f$ does not depend on time, but this is no necessity for 
the shape calculus presented in the subsequent chapter.} The topological 
setup is illustrated in Figure~\ref{fig.setting_cat}. It is in analogy 
to \cite{HHTauschCat}, but we consider an interior boundary 
$\Gamma_t$ which moves in time instead of a fixed, 
interior boundary $\Gamma_0$. 

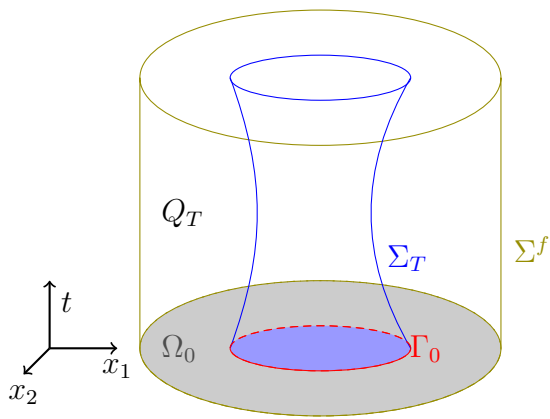
\begin{figure}[htb]
\centering
\begin{tikzpicture}[scale=0.6]
\filldraw[fill=black!20!white, draw=olive, densely dashed] (0,0) circle (4cm and 1.5cm);
\filldraw[fill=blue!40!white, draw=red, densely dashed] (0,0) circle (2cm and 0.5cm);

\draw[olive] (-4,6) -- (-4,0) arc (180:360:4cm and 1.5cm)-- (4,6) ++(-4,0) circle (4cm and 1.5cm);
\draw[densely dashed, olive] (-4,0) arc (180:0:4cm and 1.5cm);

\draw[bend left=20,blue] (-2,6) to node [auto] {} (-2,0);
\draw[red] (-2,0) arc (180:360:2cm and 0.5cm);
\draw[bend left=30, blue] (2,0) to node [auto] {} (2,6);
\draw[blue] (2,6) ++ (-2,0) circle (2cm and 0.5cm);
\draw[densely dashed, red] (-2,0) arc (180:0:2cm and 0.5cm);

\node at (1,2) (0,0mm) [label={[blue]right:$\Sigma_T$}]{};
\node at (3.8,2.2) (0,0mm) [label={[olive]right:$\Sigma^f$}]{};
\node at (1.5,0) (0,0mm) [label={[red]right:$\Gamma_0$}]{};
\node at (-4,0) (0,0mm) [label={[black!70!white]right:$\Omega_0$}]{};
\node at (-4,3) (0,0mm) [label={[black]right:$Q_T$}]{};

\begin{scope}[shift={(-6,0)}]
\draw[black, thick,->] (0,0,0) -- (1.5,0,0) node[anchor=north]{$x_1$};
\draw[black, thick,->] (0,0,0) -- (0,1.5,0) node[anchor=north west]{$t$};
\draw[black, thick,->] (0,0,0) -- (0,0,1.5) node[anchor=north]{$x_2$}; 
\end{scope}
\end{tikzpicture}
\caption{\label{fig.setting_cat}The tube $Q_T$ with the 
boundaries $\Sigma_T$ and $\Sigma^f$ for $d=2$.}
\end{figure}

For every time step $t$, we assume to have a smooth $C^2$-diffeomorphism 
$\boldsymbol\kappa$, which maps the initial domain $\Omega_0$ onto 
the time-dependent domain $\Omega_t$. In accordance with 
\cite{Moubachir}, we write
\begin{equation}\label{eq.tube_param}
\boldsymbol\kappa: [0, T ] \times \mathbb{R}^d \to \mathbb{R}^d, 
	\quad (t,\xbfm) \mapsto \boldsymbol\kappa(t, \xbfm)
\end{equation}
to emphasize the dependence of the mapping $\boldsymbol\kappa$ on 
the time, where we have $\boldsymbol\kappa(t,\Omega_0) = \Omega_t$. 
Here, $\boldsymbol\kappa \in C^2([0,T]\times \mathbb{R}^d)$ and,
as in \cite[pg.~826]{harbrecht_domain_map}, we assume the uniformity condition
\begin{equation}
\label{eq.uniformity_cond}
\|\boldsymbol\kappa(t, \xbfm) \|_{C^2( [0,T]\times \mathbb{R}^d ; \mathbb{R}^d)}, \|\boldsymbol\kappa(t, \xbfm)^{-1} \|_{C^2( [0,T] \times \mathbb{R}^d ; \mathbb{R}^d)} \leq C_{\kabfm}
\end{equation}
for some constant $C_{\kabfm} \in (0, \infty)$. To reduce the
technical level of the ensuing discussion, we assume
that $\Omega_0$ has $C^2$-smooth boundaries which implies that
the boundaries of $\Omega_t$ have the same regularity.

\begin{remark}
\label{rem.singular_val}
Notice that, due to the uniformity condition \eqref{eq.uniformity_cond}, we have as in \cite{harbrecht_domain_map}
\[ 0 < \underline{\sigma} \leq \min \lbrace \sigma (\on{D} \boldsymbol{\kappa}) \rbrace \leq \max \lbrace \sigma (\on{D} \boldsymbol{\kappa}) \rbrace \leq \overline{\sigma} < \infty,
\]
where $\sigma(.)$ denote the singular values. 
Moreover, as in \cite[Remark 1, pg.~827]{harbrecht_domain_map}, 
we assume $\det (\on{D} \boldsymbol{\kappa})$ to be positive. The smoothness
of the mapping also implies that the time 
derivative $\partial_t \boldsymbol{\kappa}$ is uniformly bounded.
\end{remark}

We shall consider the following, overdetermined initial boundary 
value problem for the heat equation, where $f$ and $g$ are defined 
at the fixed exterior boundary $\Sigma^f$
\begin{equation}\label{eq:overdetermined}
\begin{aligned}
\partial_t u &= \Delta u \ \ &&\text{in } Q_T, \\
u & = 0 \ \ &&\text{on } \Sigma_T, \\
u = f,& \ \ \frac{\partial u}{\partial \mathbf{n}} = g \ \ &&\text{on } \Sigma^f, \\
u(0, \cdot) & = 0 \ \ &&\text{in } \Omega_0.
\end{aligned}
\end{equation}
Here, $\nbfm$ denotes the normal pointing outward of the 
domain $\Omega_t$. 
In what follows, we assume that $f$ vanishes for $t=0$, which
implies the compatibility with the initial condition. 
We then seek the 
free boundary $\Sigma_T$, such that the overdetermined 
problem \eqref{eq:overdetermined} allows for a solution $u$.
In \cite[Theorem 1.1]{Chapko_et_al}, the uniqueness of such
a boundary $\Sigma_T$ is proven in the case of a time-independent 
boundary. In view of the bijective mapping $\boldsymbol\kappa$
\eqref{eq.tube_param}, this uniqueness result also holds in 
the time-dependent case. 

\subsection{Reformulation as a shape optimization problem}\label{sec.shape_problem}
The task of finding the unknown boundary $\Sigma_T$ is 
reformulated as a shape optimization problem by introducing the 
function $v$ as the solution of the initial boundary value 
problem with Dirichlet boundary conditions for the heat equation
\begin{equation}\label{eq.cat_diff_eq}
\begin{aligned}
\partial_t v & = \Delta v   && \ \ \text{in } Q_T,  \\
v & = 0 && \ \ \text{on } \Sigma_T, \\
v & = f  &&\ \ \text{on } \Sigma^f, \\
v(0, \cdot) & = 0  &&\ \ \text{in } \Omega_0. 
\end{aligned}
\end{equation}

We set $Q_0 = (0,T)\times\Omega_0$, which has 
two time-independent boundaries denoted by $\Sigma_0 :=(0,T)
\times \partial \Omega_0$.
The appropriate function spaces for parabolic problems in time
invariat domains are the
anisotropic Sobolev spaces, defined by 
\[ H^{r,s} (Q_0) := L^2\big((0,T); H^r (\Omega_0)\big) \cap H^s\big((0,T); L^2(\Omega_0)\big), 
\]
see, e.g., \cite{Chapko_et_al, Costabel,lions_magenes_v2}. Likewise,
the corresponding boundary spaces are
\[ H^{r,s} (\Sigma_0) := L^2\big((0,T); H^r (\Sigma_0)\big) \cap H^s\big((0,T); L^2(\Sigma_0)\big) 
\]
which are defined for $C^2$-boundary when $r \leq 2$. 
With these definitions at hand, we can moreover define
\begin{equation*}
\begin{aligned}
 \hat{H}^{r,s}(Q_0) & :=\big\{u = U|_{Q_0} : U \in H^{r,s}\big((0,T) \times \Omega_0\big),\ U(t, \cdot) = 0,\ t < 0\big\},\\ 
 \tilde{H}^{r,s}(Q_0) & := \big\{u = U|_{Q_0} : U \in H^{r,s}\big((0,T)  \times \Omega_0\big),\ U(t, \cdot) = 0,\ T < t \big\},\\ 
  \hat{H}^{r,s}(\Sigma_0) & := \big\{ u = U|_{\Sigma_ 0} : 
 	U \in H^{r,s}\big((0,T) \times \Sigma_0\big),\ U(t, \cdot) = 0,\ t < 0\big\}. 
\end{aligned}
\end{equation*}
As in the elliptic case, we can include also (spatial) 
zero boundary conditions into the function spaces by setting
\begin{equation*}
\begin{aligned}
\hat{H}_0^{r,s} (Q_0) :=  \big\{u \in \hat{H}^{r,s}(Q_0) : u|_{ \Sigma_0} = 0 \big\},\\ 
\tilde{H}_0^{r,s} (Q_0) :=  \big\{u \in \tilde{H}^{r,s}(Q_0) : u|_{ \Sigma_0} = 0 \big\}.
\end{aligned}
\end{equation*}
The dual spaces are denoted by $r,s \leq 0$ and we especially have
\[ 
  \hat{H} ^{-r, -s} (Q_0) = \big[\tilde{H}_0^{r,s} (Q_0)\big]' \quad \text{for } r - \frac{1}{2} \notin \mathbb{Z}. 
\]
Finally, we introduce the test space
\begin{equation}\label{eq.space_V}
V(Q_0) := \big\{v = U|_{Q_0}: U \in C_0^2\big((- \infty ,T) \times \Omega_0\big)\big\}
\end{equation}
as in \cite{Chapko_et_al}, which is a dense subspace of $\tilde{H}_0^{1, \frac{1}{2}} (Q_0)$ \cite{Chapko_et_al} 
(for a $C^{\infty}$-boundary, see for example \cite[Remark 2.2 on pg.~8]{lions2}).

We are now in the position to introduce the 
non-cylindrical analogues of the above spaces by setting
\[H^{r,s}(Q_T) := \lbrace v \in L^2(Q_T): v \circ \kabfm \in H^{r,s} (Q_0) \rbrace \]
and likewise for all the other spaces, where the composition with $ \boldsymbol\kappa$ 
only acts on the spatial component. Due to the chain rule, $v\circ\boldsymbol\kappa$ 
and $v$ have the same Sobolev regularity, provided that the mapping $\boldsymbol\kappa$ 
is smooth enough, see for example \cite[Theorem 3.23]{McLean} for the elliptic case. 
We especially have the equivalence of norms for $|s| \leq 2$ 
\begin{equation}\label{eq.equivalence_norm}
\| v \circ \kabfm \|_{H^s(\Omega_0)} \sim \| v \|_{H^s(\Omega_t)}.
\end{equation}

For the cylindrical case it is well known that the solution 
operator $f \mapsto S_0f := u$ of Dirichlet problem of the 
heat equation
\begin{alignat*}{3}
(\partial_t - \Delta ) u & = 0\quad && \text{in } Q_0, \\
 u & = f\quad  && \text{on } \Sigma_0.
\end{alignat*}
with homogeneous initial conditions is an isomorphism between the spaces
\[
S_0: \hat H^{\frac{1}{2}+s,\left(\frac{1}{2}+s\right)/2}(\Sigma_0)
\to \hat{H}^{1,\frac{1}{2}}(Q_0)
\]
for $s>-\frac{1}{2}$ when $\Omega_0$ is smooth and for $|s|<\frac{1}{2}$ when
$\Omega_0$ is Lipschitz, see \cite[Theorem 5.3]{lions_magenes_v2} and
\cite[Proposition 4.13]{Costabel}. 

For the existence, uniqueness and regularity of solutions to \eqref{eq.cat_diff_eq},
we have to make sure the analogous result also holds on a non-cylindrical 
domain. The main techique of the argument is to transport the heat
equation to a parabolic problem with variable coefficients in the
space-time cylinder $Q_0$ and apply the same functional analytic tools of
the above references there. 

\begin{theorem}
\label{thm.solution_operator_isomorphism}
There exists a unique solution $v \in \hat{H}^{1,\frac{1}{2}} (Q_T)$ 
satisfying the boundary condition in \eqref{eq.cat_diff_eq} and
\begin{equation}\label{eq.weak_form_pde}
S(v, \varphi) := \int_0^T\int_{\Omega_t} \{\nabla v \cdot \nabla \varphi + \partial_t v \varphi\}\, \drm \xbfm \drm t
	= 0\ \text{for all}\ \varphi \in \tilde{H}^{1, \frac{1}{2}}_0(Q_T).
\end{equation}
\end{theorem}

\begin{proof}
The assertion follows if we can show existence and uniqueness 
of the solution to the following generalization of problem 
\eqref{eq.cat_diff_eq}
\begin{equation}
\label{eq.generalized_diff_eq}
\begin{aligned}
(\partial_t - \Delta) v & = h   && \ \ \text{in } Q_T,  \\
 v & = f && \ \ \text{on } \Sigma_T \cup \Sigma^f, \\
v(0, \cdot) & = 0   &&\ \ \text{in } \Omega_0.
\end{aligned}
\end{equation}
Its weak formulation reads
\begin{equation}\label{eq.weak_wo_reynold}
  S(v,u) = \int_0^T \int_{\Omega_t} hu\,\d{\bf x}\d t,
\end{equation}
where $S$ is given by \eqref{eq.weak_form_pde}. We set 
$u^t = u \circ \kabfm$ and similarly for $v^t$ and $h^t$.

Transforming \eqref{eq.weak_wo_reynold} back to $Q_0$ 
by using Lemma \ref{lem.transport_eq_back_0} with $\xibfm = \kabfm$, 
$Q_{\varsigma} = Q_T$ and $Q_{\tau} = Q_0$ gives
\[ \int_0 ^T \langle \partial_t v^t(t), u^t(t) \rangle_{L^2(\Omega_0)} 
+ a(t; v^t(t), u^t(t)) \, \drm t = \int_0 ^T \langle h^t(t), u^t(t)\rangle_{L^2 (\Omega_0)} \, \drm t, 
\]
where $a$ is defined in Lemma \ref{lem.transport_eq_back_0}.

To show solvabilty of \eqref{eq.weak_wo_reynold} we apply
\cite[Chapter 3, Theorem 4.1.]{lions_magenes_v1}, which requires
boundedness and  coercivity of $a$.
The boundedness follows easily from Remark \ref{rem.singular_val}. 
It remains to show coercivity,  that is, there exist 
some constants $\alpha > 0$, $\lambda \in \mathbb{R}$, such 
that for almost all $t \in (0,T)$
\begin{equation}\label{eq:coercivity}
  a(t; u^t, u^t) \geq \alpha \|u^t\|^2_{H^1(\Omega_0)} 
	- \lambda \|u^t \|^2_{L^2(\Omega_0)}
\end{equation}
holds for all $u^t \in H^1_0(\Omega_0)$. With the help 
of the Cauchy-Schwarz inequality, we have
\begin{align*}
a(t; u^t, u^t) &\geq \int_{\Omega_0} \|(\on{D} \kabfm)^{-\intercal} \nabla u^t \|^2 \, \drm \xbfm \\
& \quad - \int_{\Omega_0} \bigg\| \Big(\underbrace{(\on{D} \kabfm)^{-\intercal}  
	\frac{1}{\det (\on{D} \kabfm)} \nabla \big( \det (\on{D} \kabfm)\big)}_{:=a_1} 
	+ \underbrace{\partial_t \kabfm}_{:=a_2} \Big) u^t \bigg\| \| (\on{D} \kabfm)^{-\intercal} \nabla u^t \| \, \drm \xbfm.
\end{align*}
Completing the square gives
\begin{align*}
a(t; u^t, u^t) & \geq \underbrace{\int_{\Omega_0} \frac{1}{2} \Big( \|(\on{D} \kabfm)^{-\intercal} \nabla u^t \| - \|(a_1+a_2)u^t\| \Big)^2}_{\geq 0} \, \drm \xbfm \\
& + \int_{\Omega_0} \frac{1}{2} \|(\on{D} \kabfm)^{-\intercal} \nabla u^t \|^2 \, \drm \xbfm - \int_{\Omega_0} \frac{1}{2} \|(a_1+a_2)u^t\|^2 \, \drm \xbfm.
\end{align*}
Discarding the positive term and due to 
Remark \ref{rem.singular_val}, we have
\begin{align*}
a(t; u^t, u^t) & \geq C |u^t|^2_{H^1(\Omega_0)} - \frac{1}{2} \int_{\Omega_0} |u^t|^2 \| a_1 + a_2 \|^2 \, \drm \xbfm.
\end{align*}
and, therefore, by using the parallelogram law
\begin{align*}
a(t; u^t, u^t) & \geq C |u^t|^2_{H^1(\Omega_0)} -  \int_{\Omega_0} |u^t|^2 ( \| a_1 \|^2 + \| a_2 \|^2) \, \drm \xbfm.
\end{align*}
Now we can apply again Remark \ref{rem.singular_val} to $a_1$ and $a_2$ and 
the Poincar{\'e}-Friedrichs inequality to the first term to arrive at the desired 
estimate \eqref{eq:coercivity}.

Secondly, following the lines of \cite{Costabel}, the analogue of 
\cite[Lemma 2.8]{Costabel} reads: For every $h \in \hat{H}^{-1, - \frac{1}{2}} (Q_T)$, 
there exists a unique $v \in \hat{H}^{1, \frac{1}{2}}_0(Q_T)$ satisfying 
$(\partial_t - \Delta) v = h$ in $Q_T$. For the proof, we can straightforwardly 
modify the proof \cite[Lemma 2.8]{Costabel}, which uses the adjoint operator 
and interpolation results. 

Thirdly, due to the surjectivity of the trace operator, we can then follow 
the proof of \cite[Theorem 2.9]{Costabel} to obtain the statement in the 
theorem.
\end{proof}

For the given state equation \eqref{eq.cat_diff_eq},
we introduce the tracking-type functional for the 
Neumann data at the fixed boundary $\Sigma^f$
\begin{equation}\label{eq.cat_functional_neumann_tracking}
J (\Sigma_T) = \frac{1}{2} \int_0^T \int_{\Gamma^f} 
	\left( \frac{\partial v}{\partial \nbfm} - g \right)^2 \, \drm \sigma \drm t.
\end{equation}
This objective functional should be minimized in the space of admissible 
boundaries $\Sigma_T$. It is nonnegative, and it is zero and hence 
minimal if and only if $v=u$. The objective functional measures the 
$L^2$-error of the data mismatch and thus corresponds to the 
minimization in the least-squares sense.

\section{Computation of the shape derivative} \label{sec.shape_calc_neumann}
\subsection{Shape calculus}
In order to minimize the objective functional \eqref{eq.cat_functional_neumann_tracking}, 
we apply a gradient-based optimization method. To this end, 
we shall compute the shape derivative of the functional. 

The shape calculus for time-dependent problems has been 
formulated by means of the speed method in \cite{Dziri} and 
\cite{Moubachir}. The speed method allows for deformations 
which are not only small perturbations of the domain. One 
intends to find a velocity field $\Vbfm$, which generates the 
optimal tube. The solution $\ALEoV{t}{\cdot} : \xbfm \mapsto 
\xbfm_t = \ALEoV{t}{x}$ of the differential equation 
\cite[pg.~6]{zolesio_identification}
\[
\begin{aligned}
\frac{\partial }{\partial t} \ALEoV{t}{x}  & = \Vbfm\big(t, \ALEoV{t}{x}\big) && \ \ \text{in }  (0,T) \times \Omega_0,  \\
\ALEoV{0}{x} & = \xbfm && \ \ \text{in } \Omega_0
\end{aligned}
\]
describes the pathline of an individual particle being exposed to 
the velocity field $\Vbfm$. Hence, if we would inject a drop of dye 
at a certain point and time, and we do a time-lapse photography, 
we would see the pathline \cite{spurk}. In other words, when 
considering $t$ as the trajectory parameter, a fixed point $\xbfm$ 
gets moved along the trajectory $\xbfm_t = \ALEoV{t}{x}$. The 
point $\xbfm$ can be thought of as the Lagrangian (or material) 
coordinate, while $\xbfm_t$ is the Eulerian (field) coordinate 
\cite[pg.~49]{Soko_Zolesio}. The speed method is favorable 
when considering the Eulerian setting \cite{Moubachir}.

For the Lagrangian setting, which we consider here, the perturbation 
of identity is preferable. The shape calculus for the perturbation 
of identity is briefly stated in \cite{Moubachir} as well. For our 
computations, we shall exploit the bijective mapping $\boldsymbol
\kappa$ from \eqref{eq.tube_param}, which implies the mapping 
scheme displayed in Figure~\ref{fig.mapping_scheme_perturb_id}.
With the mapping $\boldsymbol\kappa$ we can associate 
the velocity field 
\begin{equation}\label{eq.velocity_field}
\Vbfm = \partial_t \boldsymbol\kappa \circ \boldsymbol\kappa^{-1},
\end{equation}
which could be used for the speed method. Since the outer 
boundary $\Sigma^f$ of the tube is fixed, this vector field is 
zero in normal direction there. 

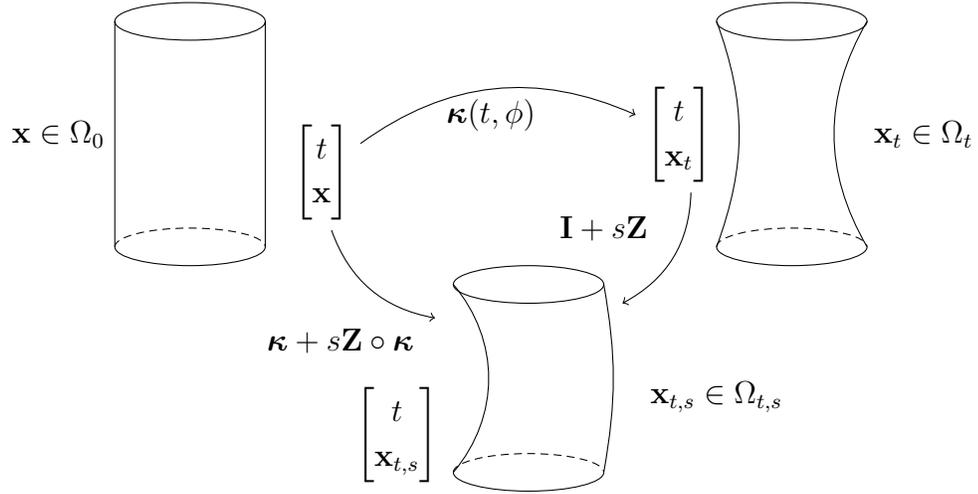
\begin{figure}[htb]
\centering
\begin{tikzpicture}[scale=0.5]
\draw (-2,6) -- (-2,0) arc (180:360:2cm and 0.5cm) -- (2,6) ++ (-2,0) circle (2cm and 0.5cm);
\draw[densely dashed] (-2,0) arc (180:0:2cm and 0.5cm);

\draw[bend left=20] (14,6) to node [auto] {} (14,0);
\draw (14,0) arc (180:360:2cm and 0.5cm);
\draw[bend left=30] (18,0) to node [auto] {} (18,6);
\draw (18,6) ++ (-2,0) circle (2cm and 0.5cm);
\draw[densely dashed] (14,0) arc (180:0:2cm and 0.5cm);

\node (C) at (6.8,-2) {};

\draw[bend left=40] (7,-1) to node [auto] {} (7,-6);
\draw (7,-6) arc (180:360:2cm and 0.5cm);
\draw[bend right=10] (11,-6) to node [auto] {} (11,-1);
\draw (11,-1) ++ (-2,0) circle (2cm and 0.5cm);
\draw[densely dashed] (7,-6) arc (180:0:2cm and 0.5cm);

\node (A) at (3.5,2) {$\begin{bmatrix} t \\ \xbfm \end{bmatrix}$};
\node (B) at (13,3) {$\begin{bmatrix} t \\ \xbfm_t \end{bmatrix}$};
\node (CC) at (5.5,-5) {$\begin{bmatrix} t \\ \xbfm_{t,s} \end{bmatrix}$};
\draw[->,  bend left=30] (A) to (B);
\node (E) at (8,3.5) {$\boldsymbol\kappa (t, \phi) $};

\draw[->,  bend right=30] (A) to (C);
\node at (4,-2.5) {$\boldsymbol\kappa + s \Zbfm \circ \boldsymbol\kappa $};
\draw[->,  bend left=30] (B) to (11.5,-1.5);
\node at (11,0.5) {$\mathbf{I} + s \Zbfm $};

\node (G) at (19.5,3) {$\xbfm_t \in \Omega_t$};
\node (H) at (14,-4) {$\xbfm_{t,s} \in \Omega_{t,s}$};
\node (HH) at (-3.5,3) {$\xbfm \in \Omega_0$};
\end{tikzpicture}
\caption{Perturbation of identity in the Lagrangian setting.}
\label{fig.mapping_scheme_perturb_id}
\end{figure}

In order to apply the traditional shape calculus, 
we would like 
to perturb the tube. To this end, we consider a vector field 
$\Zbfm (t, \xbfm)$, which generates the perturbation of identity 
$\mathbf{I} + s \Zbfm$. It yields a new tube 
\[
  Q_T^s = \bigcup_{0<t<T} \big(\{t\}\times (\mathbf{I}+s\Zbfm)(\Omega_t)\big).
\]
Notice that the perturbations under consideration are 
horizontal, meaning that we consider perturbations of 
$(t,\boldsymbol\kappa)$ of the type $(0,\Zbfm)$, 
compare \cite{Moubachir}. Moreover, $\mathbf{I} + s\Zbfm$ should satisfy a uniformity condition as in \eqref{eq.uniformity_cond}. 

\subsection{Local shape derivative}
As in the time-independent case, we can define non-cylindrical 
material and local shape derivatives. The material derivative 
$\dot{v} [\Zbfm]$ is defined as 
\[ 
  \dot{v} [\Zbfm] = \lim_{s \to 0} \frac{v_{t,s} (t, \cdot) \circ (\mathbf{I} + s \Zbfm) - v_t}{s}, 
\]
while the local shape derivative $\delta v = \delta v [\Zbfm]$ in the 
direction $\Zbfm$ is given by
\[ 
\delta v [\Zbfm] = \lim_{s \to 0} \frac{v_{t,s} (t, \cdot) - v_t}{s}. 
\]
Here, $v_{t,s}$ denotes the state computed on the perturbed 
domain $Q_T^s$ and $v_t$ the state computed on $Q_T$,
see \cite[pg.~166]{Moubachir}. These two non-cylindrical 
derivatives are connected by the relation
\[ 
  \delta v [\Zbfm] = \dot{v} [\Zbfm] - \nabla v \cdot \Zbfm. 
\]

\begin{theorem} \label{thm.local_shape_der}
The local shape derivative  of the state $v$ from 
\eqref{eq.cat_diff_eq} can be computed as the solution 
of the partial differential equation
\begin{equation} \label{eq.local_shape_der}
\begin{aligned}
\partial_t \delta v & = \Delta \delta v  && \ \ \text{in } Q_T,\\
\delta v & = - \langle \mathbf{Z}, \mathbf{n} \rangle \frac{\partial v}{\partial \mathbf{n}}  && \ \ \text{on } \Sigma_T, \\
\delta v & = 0 &&\ \ \text{on } \Sigma^f, \\
\delta v(0, \cdot) & = 0 &&\ \ \text{in } \Omega_0.
\end{aligned}
\end{equation}
\end{theorem}

The proof of the local shape derivative is presented in 
Appendix \ref{sec.proof_local_shape}, where we reformulate 
the time-indepen\-dent proof found in \cite{Chapko_et_al} for 
the time-dependent setting.

\subsection{Shape derivative of the objective functional}
With the local shape derivative at hand, we are in the 
position to compute the shape derivative of the objective 
functional \eqref{eq.cat_functional_neumann_tracking}, 
which is defined by
\[ 
 \nabla J (Q_T)[\Zbfm] = \lim_{s \to 0} \frac{ J\left(Q_T^s \right) - J\left(Q_T\right)}{s}. 
\]

\begin{theorem}
The shape derivative of the objective functional 
\eqref{eq.cat_functional_neumann_tracking} in the 
direction $\Zbfm (t, \xbfm) \in C^2(\Sigma_T)$ reads
\begin{equation}\label{eq.shape_grad_neumann_tracking}
  \nabla J(Q_T) [\Zbfm] = - \int_0^T \int_{\Gamma_t} \frac{\partial p}{\partial \nbfm} 
  	\frac{\partial v}{\partial \nbfm} \langle \Zbfm, \nbfm \rangle \, \drm \sigma \drm t,
\end{equation}
where the adjoint state $p$ satisfies also the 
heat equation, but reversal in time:
\begin{equation}\label{eq.adjoint_problem}
\begin{aligned}
-\partial_t p & =  \Delta p \ \ && \text{in } Q_T, \\
p & =0 \ \ && \text{on } \Sigma_T, \\
p & = \frac{\partial v}{\partial \nbfm} - g \ \ && \text{on } \Sigma^f, \\
p(T, \cdot) & = 0 \ \ && \text{in } \Omega_T.
\end{aligned}
\end{equation}
\end{theorem}

\begin{proof}
Since we are not perturbing the exterior boundary, we have
$\Zbfm = {\bf 0}$ in a neighborhood of $\Sigma^f$. Therefore, 
we conclude
\[ 
  \nabla J (Q_T) [\Zbfm] = \int_0^T \int_{\Gamma^f} \frac{\partial \delta v}{\partial \nbfm} 
  	\left(\frac{\partial v}{\partial \nbfm} - g \right) \, \drm \sigma \drm t. 
\]
In view of the adjoint state equation \eqref{eq.adjoint_problem}, 
we can reformulate the derivative of $J$ by
\[
  \nabla J (Q_T)[\Zbfm]  =  \int_0^T \int_{\Gamma^f} 
  	p \frac{\partial \delta v }{\partial \nbfm} \, \drm \sigma \drm t.
\]
To derive \eqref{eq.shape_grad_neumann_tracking}, we 
apply Green's theorem and obtain
\begin{align*}
0 & = \int_0 ^T \int_{\Omega_t} \big\{(\partial_t \delta v- \Delta \delta v) p 
	+ \delta v (\partial_t p + \Delta p)\big\} \, \drm \xbfm \drm t \\
&= \int_0 ^T \int_{\Omega_t} \partial_t (\delta vp) \, \drm \xbfm \drm t 
	+ \int_0 ^T \int_{\Gamma_t \cup\Gamma^f} \bigg\{\frac{\partial p}{\partial \nbfm} \delta v
	- \frac{\partial \delta v}{\partial \nbfm} p\bigg\} \, \drm \sigma \drm t.
\end{align*}
Since the integrands are smooth enough, we can apply the
Reynolds transport theorem (see \cite[pg.~78]{Gurtin} for example) 
to treat the domain integral. Recall that the velocity $\Vbfm$, which
transports the initial domain through the space-time tube, is given 
by \eqref{eq.velocity_field}. In combination with the end and 
initial conditions of $p$ and $\delta v$, respectively, we thus 
obtain
\begin{align*}
0 & = \underbrace{\int_0 ^T \frac{\drm}{\drm t} \int_{\Omega_t} 
	\delta vp\,\drm \xbfm \drm t}_{=\,0} - \int_0 ^T \int_{\Gamma^f \cup\Gamma_t} 
	\underbrace{\delta v}_{=\,0\,\text{on}\,\Gamma^f}\underbrace{p}_{=\,0\,\text{on}\,\Gamma_t}
	\langle\Vbfm, \nbfm \rangle\, \drm \sigma \drm t \\
  & \hspace*{5cm} + \int_0 ^T \int_{\Gamma_t} \frac{\partial p }{\partial \nbfm} 
	\delta v\,\drm \sigma \drm t - \int_0 ^T \int_{\Gamma^f}
	\frac{\partial \delta v}{\partial \nbfm} p \, \drm \sigma \drm t,
\end{align*}
that is
\[
  \int_0 ^T \int_{\Gamma^f} \frac{\partial \delta v}{\partial \nbfm} p \, \drm \sigma \drm t
	= \int_0 ^T \int_{\Gamma_t} \frac{\partial p }{\partial \nbfm} \delta v\,\drm \sigma \drm t.
\]
Hence, by inserting the boundary condition for 
$\delta v$, we finally arrive at the desired result 
\eqref{eq.shape_grad_neumann_tracking}.
\end{proof}

Note that the tracking-type functional for the Dirichlet data 
has been considered in the setting of the speed method in 
\cite[pg.~36--46]{Moubachir}. It leads also to the same 
local shape derivative and shape gradient as in the 
time-independent case derived in \cite{HHTauschCat}. 
This is thus consistent with the formulae stated here in 
case of the tracking-type functional for the Neumann data.

\section{Discretization of the shape optimization problem}\label{sec.discretization_shape_opt}
For our numerical experiments, we consider a two-dimensional spatial 
domain with a star-shaped void. As only its boundary is of interest and 
the shape gradient is also defined as a boundary integral, it suffices to 
parametrize just the interior boundary. Moreover, we consider only boundary 
perturbation fields $\Zbfm$, because these are the only relevant 
perturbation fields as \eqref{eq.shape_grad_neumann_tracking} 
shows. 

Our choice of parametrization of the interior moving boundary 
$\Sigma_T$ of $Q_T$ is 
\[ 
\Sigma_T = \left\{
\begin{bmatrix}
t \\ \boldsymbol\gamma(t,\phi)
\end{bmatrix}\in\mathbb{R}^3: t\in [0,T],\ \phi\in [0,2\pi)\right\},
\]
where the time-dependent parametrization 
$\boldsymbol\gamma(t,\cdot):[0,2\pi)\to\Gamma_t$ employs
polar coordinates
\[ \boldsymbol\gamma(t,\phi) =  w(t, \phi) \begin{bmatrix}
\cos( \phi) \\
\sin(\phi)
\end{bmatrix}.\]
Here, $w(t, \phi)$ denotes the time- and angle-dependent radius, 
given by
\[
w(t,\phi) := \sum_{\ell=0}^{N_L} L_{\ell} (t) \Bigg(\underbrace{
\alpha_{0,\ell} +  \sum_{k=1}^{N_K-1}\big\{\alpha_{k,\ell}  \cos (k \phi) 
+ \beta_{k,\ell}\sin(k \phi)\big\}+\alpha_{N_K,\ell}\cos(N_K\phi)}_{=: \omega (\phi)}\Bigg), 
\]
with $L_\ell(t)$ being appropriate dilations and translations 
of the Legendre polynomials of degree $\ell$.

Finding the optimal tube now corresponds to determining the 
unknown coefficients $\alpha_{k,\ell}$ and $\beta_{k,\ell}$ of the 
parametrization. Hence, we have the following finite dimensional 
problem:
\[
\text{Seek } \boldsymbol\gamma^\star \in Z_N 
	\text{ such that } \nabla J (\boldsymbol\gamma^\star)[\Zbfm] = 0 
		\ \text{for all } \Zbfm \in Z_N.
\]
Here, $Z_N$ is the finite dimensional ansatz space of 
parametrizations. To compute the discrete shape gradient, 
we hence have to consider the directions
\begin{equation}\label{eq.choice_of_Z1}
(\Zbfm \circ\boldsymbol\gamma)(t,\phi) = L_\ell(t) \cos (k\phi) \begin{bmatrix}
\cos(\phi) \\ \sin(\phi)
\end{bmatrix}
\end{equation}
for all $\ell=0, \dots, N_L$ and $k = 0, \dots,N_K$, and
\begin{equation}\label{eq.choice_of_Z2}
(\Zbfm \circ\boldsymbol\gamma)(t,\phi) = L_\ell(t) \sin (k\phi) \begin{bmatrix}
\cos(\phi) \\ \sin(\phi)\end{bmatrix}
\end{equation}
for all $\ell=0, \dots, N_L$ and $k = 1, \dots, N_K-1$.

With the specific parametrization at hand, the discrete shape 
gradient with respect to the parameters $t$ and $\phi$ reads
\begin{equation}\label{eq.shape_grad_param}
 \nabla J (Q_T) = \int_0^T \int_0^{2 \pi}\bigg(\frac{\partial p}{\partial \nbfm}\circ\boldsymbol\gamma\bigg) 
 \bigg(\frac{\partial v}{\partial \nbfm}\circ\boldsymbol\gamma\bigg)
 \left[\begin{smallmatrix} L_1(t)\\\vdots\\ L_{N_L}(t)\end{smallmatrix}\right]\otimes
 \left[\begin{smallmatrix}
 \sin((N_K-1)\phi)\\ \vdots\\ \sin(\phi)\\ 1 \\ \cos(\phi) \\ \vdots \\ \cos(N_K\phi)\end{smallmatrix}\right]
  w(t, \phi) \, \drm \phi \drm t, 
\end{equation}
compare \eqref{eq.shape_grad_neumann_tracking}, where 
we plugged in the choices for the perturbation fields \eqref{eq.choice_of_Z1} 
and \eqref{eq.choice_of_Z2}, respectively, and used the parametrization 
$\boldsymbol\gamma$ to compute the normal $\nbfm$.

The integral in the shape gradient \eqref{eq.shape_grad_neumann_tracking} 
is computed by using a trapezoidal rule in space and a trapezoidal rule
with a singularity correction at the endpoint $t = T$ in time (see the
next section for details). The Legendre polynomials are computed 
by using their three term recurrence formula as described in 
\cite{numerical_recipes}, and are normalized afterwards while 
the Fourier series is evaluated efficiently by the fast Fourier transform.

The gradient-based method of our choice is the quasi Newton 
method, updated by the inverse BFGS rule without damping,
cf.~\cite{Geiger_Kanzow_unres}. A second order line search is 
applied to find an appropriate step size in the quasi Newton method. 
For an overview of possible other optimization algorithms 
in general, see \cite{Dennis_Schnabel, Fletcher}.

\section{Solving parabolic boundary value problems}\label{sec.solving_parabolic_BVP}
We briefly describe the numerical method for solving the state and
adjoint equation by using a boundary integral formulation. Since this is
the approach that was already taken in~\cite{HHTauschCat} for a fixed
boundary, we focus in this section on the changes for the time
dependent case.

Both, the state and the adjoint equation, are Dirichlet problems of 
the heat equation with homogeneous initial conditions. In the case 
of the adjoint equation this becomes apparent after the change of 
variables $t \mapsto T - t$.

The boundary integral approach has distinct advantages over domain
based approaches, because it is not necessary to mesh a time dependent
domain or consider the transported problem in a cylindrical domain. 
Instead, we solve the Green's integral equation.  For a
time-dependent boundary, it has the form
\begin{equation}\label{eq:greenrep}
\frac{1}{2} \phi(t,\xbfm) = \V \gamma_1^-\phi(t,\xbfm)
               - \K \phi(t,\xbfm), 
\quad (t,\xbfm) \in \Sigma_T \cup \Sigma^f.
\end{equation}
Here, $\V$ and $\K$ are the thermal single and double layer operators
defined below, and $\phi$ is a solution to the source-free heat equation 
with homogeneous initial conditions. Time dependence of the surface 
appears in the normal trace, which is defined as
\begin{equation}\label{normaltrace}
\gamma_1^\pm \phi 
:= \frac{\partial \phi}{\partial \nbfm}  \mp \frac{1}{2}\langle{\bf V},{\bf n}\rangle\phi,
\end{equation}
where $\langle{\bf V},{\bf n}\rangle$ is the normal velocity of the
surface. The extra term in the definition of $\gamma_1^\pm$ arises
from the Reynolds transport theorem in the derivation of
\eqref{eq:greenrep}. Details can be found in \cite{tausch18}.

For the discretization of \eqref{eq:greenrep}, it is desirable to have
a method that can be easily adapted to time-dependent geometries,
hence we use the Nystr\"om discretization method of
\cite{tausch09a,tausch18}. To that end, we write the thermal layer
potentials in the form
\begin{align}
\V\phi(t,\xbfm) &= \frac{1}{\sqrt{4\pi}}\int\limits_0^t 
      \frac{1}{\sqrt{t-\tau}} V\phi (t, \tau, \xbfm)
      \d\tau,\label{def:V} \\
\K\phi(t,\xbfm) &= \frac{1}{\sqrt{4\pi}}\int\limits_0^t 
      \frac{1}{\sqrt{t-\tau}} K\phi (t,\tau, \xbfm)
      \d\tau,\label{def:K}
\end{align}
where
\begin{align}
V\phi(t,\tau,\xbfm) &= \int\limits_{\Gamma_\tau\cup\Gamma^f} \frac{1}{(4\pi(t-\tau))^{\frac{d}{2}}} 
  \exp\left( -\frac{\|\xbfm - \ybfm\|^2}{4 (t-\tau)}\right)
  \phi(\ybfm,\tau)\d\sigma_{\bf y}, \label{def:poissonV}\\
K\phi(t,\tau,\xbfm) &= \int\limits_{\Gamma_\tau\cup\Gamma^f} \frac{1}{(4\pi(t-\tau))^{\frac{d}{2}}}
 \,\gamma_{1,y}^+\left[ \exp\left( -\frac{\|\xbfm - \ybfm\|^2}{4 (t-\tau)}\right)\right]
  \phi(\ybfm,\tau)\d\sigma_{\bf y},\label{def:poissonK}
\end{align}
and $\Gamma_\tau\cup\Gamma^f = \partial\Omega_\tau$, i.e., the union of the free
and the fixed boundary. 

The kernel in the above time-dependent surface potentials is the
Green's function of the ($d-1$)-dimensional heat equation. Thus, 
they may be regarded as Poisson-Weierstrass integrals defined 
on a surface instead of the usual plane. As in the planar case, these
integrals are smooth functions in all variables when $0\leq \tau\leq t$.
The limiting behavior of these functions as $\tau\to t$ is
\begin{equation}\label{eq.asymptotics}
\begin{aligned}
V\phi(t, \tau,\xbfm)
 &= \phi(t, \tau,\xbfm) + \mathcal{O}(t), \\
K\phi(t, \tau,\xbfm)
 &= H(t, \xbfm) \phi(\xbfm) + \mathcal{O}(t) ,
\end{aligned}
\end{equation}
where $H(\cdot)$ is the mean curvature of the surface 
$\Gamma_t\cup\Gamma^f$, see~\cite{tausch18}.

Since the functions $V\phi$ and $K\phi$ are smooth, the integral operators in 
\eqref{def:V} and \eqref{def:K} have a $(t-\tau)^{-1/2}$
singularity, which suggests to use the trapezoidal rule with a
singularity correction at the endpoint $t=\tau$. It is
shown in \cite{tausch09a} that the rule
\begin{equation}\label{trapezoidal:desing0}
\V\phi(\xbfm,t_n) = {h\over \sqrt{4\pi}} \sum_{j=0}^{n-1}{\!'}
{1\over \sqrt{t_n-t_j}} V(t_n,t_j) \phi(\xbfm,t_j) + 
     \mu_n \psi(\xbfm,t_n) + \epsilon_h,
\end{equation}
where $h$ is the time step length, $t_j = hj$ and
$$
\mu_n = \sqrt{t_n\over \pi} - 
{h\over \sqrt{4\pi}}\sum_{j=0}^{n-1}{\!'} {1\over \sqrt{t_n-t_j}},
$$
has a quadrature error of $\epsilon_h =  \mathcal{O}(h^{3/2})$. Here, 
the prime at the summation sign indicates that the $j=0$ term in the sum
is multiplied by the factor $1/2$. For the double layer analogous result
holds when the $\mu_n$-term is mulitplied by the curvature.  A fully
discrete version is obtained by approximating the surface integrals in
\eqref{def:poissonV} and \eqref{def:poissonK} by a surface quadrature
rule, usually a composite rule that integrates polynomials on
triangular patches exactly. If the spatial mesh width $h_s$ satisfies
$\sqrt{h_s}\sim h$ and the spatial rule has at least degree of
precision two then the quadrature error in \eqref{trapezoidal:desing0}
can be preserved, see \cite{tausch09a}. In the time dependent case, these
rules are constructed on $\Gamma_0\cup\Gamma^f$ and then mapped 
to $\Gamma_t\cup\Gamma^f$.

For the state equation, the solution is smooth and the Nystr\"om method
based on the above quadrature is used to computed the normal trace of 
the solution. Thus the Neumann data at the quadrature nodes is computed 
from \eqref{eq:greenrep} by substituting the given Dirichlet data of 
\eqref{eq.cat_diff_eq}. This gives approximate values of the shape 
functional \eqref{eq.cat_functional_neumann_tracking} and the 
boundary condition in the adjoint state \eqref{eq.adjoint_problem}.

The next task is to compute the Neumann data in the shape gradient
\eqref{eq.shape_grad_neumann_tracking} by solving the adjoint state. 
As already observed in \cite{HHTauschCat}, the adjoint equation (after
time transformation $t \mapsto T-t$) has a singularity at $\tau=0$
because the homogeneous initial condition is not compatible with the
in general non vanishing Dirichlet condition at $t=0$. 

It can be concluded from
\eqref{eq.asymptotics} and Green's integral equation that the Neumann
data has a $t^{-1/2}$-singularity at $t=0$. To preserve the
$\mathcal{O}(h^{\frac{3}{2}})$ accuracy, the time quadrature
rule \eqref{trapezoidal:desing0} must be modified with singularity
corrections on both endpoints. Since the normal velocity of the
boundary does not appear in \eqref{eq.asymptotics}, the derivation 
and the weights of this rule are identical to the case of a steady
boundary. Since this can be found in \cite{HHTauschCat}, it is not
repeated here.

\section{Numerical experiments}\label{sec.results}
The exterior, fixed boundary is chosen as the mantle of the 
cylinder with radius 1, where its height corresponds to the 
time interval $(0,T) = (0,1)$. We choose $N_t = 90$ time 
intervals and, for every time step, $N_{\bf x} = 80$ spatial points.
The void is depicted in Figure~\ref{fig.shape_desired}, where
the time corresponds to the $z$-axis. It is discretized by the 
same number of time intervals and spatial points as the 
exterior boundary.
 
\begin{figure}[htb]
\centering
\includegraphics[width=0.7\textwidth]{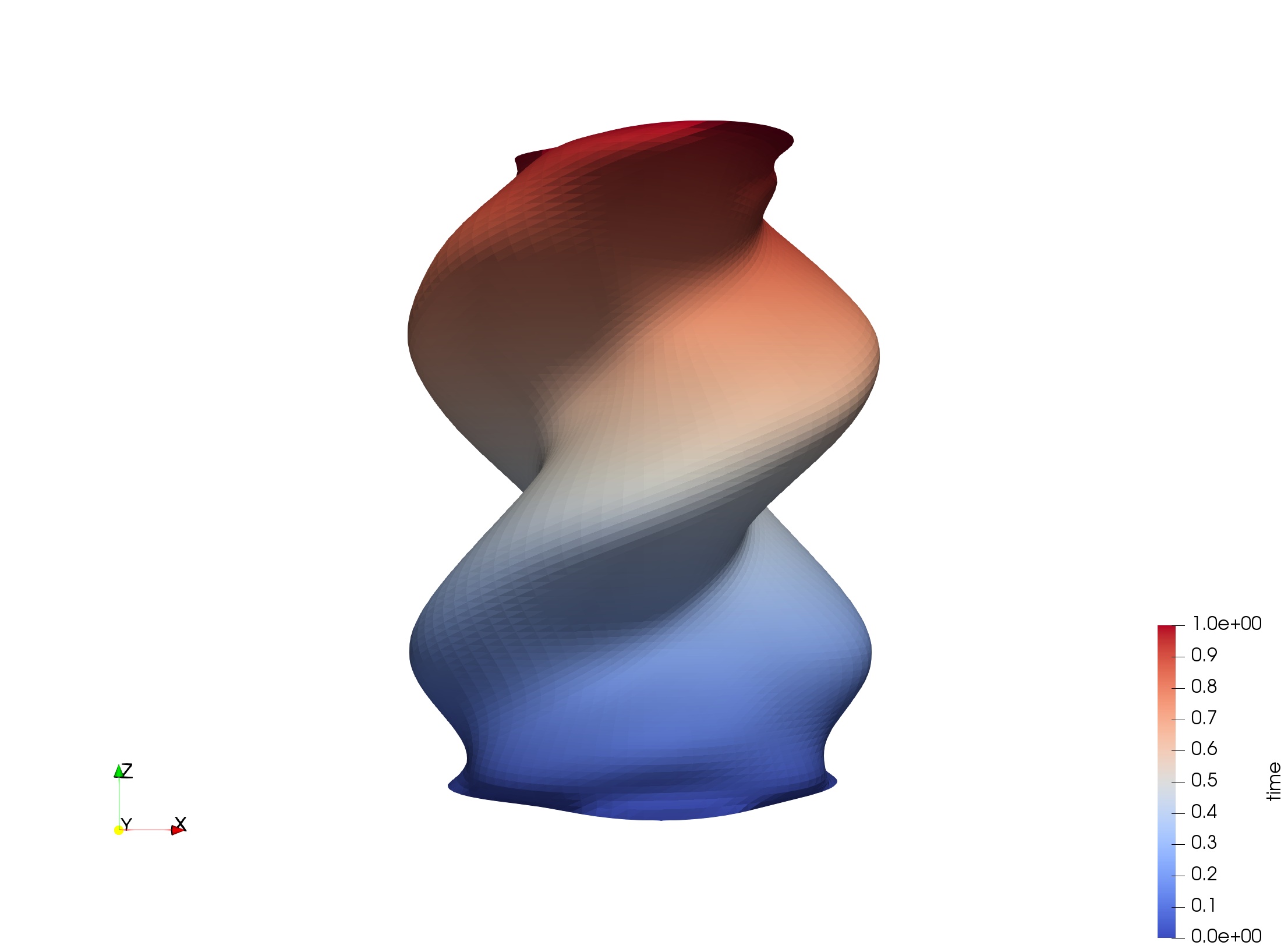}
\caption{\label{fig.shape_desired} Given inclusion in space and time.}
\end{figure}

We first solve the forward problem to construct the desired Neumann
data $g$. To this end, we choose the desired shape found in
Figure~\ref{fig.shape_desired} and choose the Dirichlet data
$f(t,\cdot) = t$, which matches with the initial data $u(0,\cdot) = 0$
in $\Omega_0$. In order to avoid an inverse crime, we use an 
indirect boundary element approach by solving the thermal single 
layer equation and then recover the Neumann data by applying
the thermal adjoint operator. In addition, we add $1\%$ random 
noise to the synthetic data.

Now, we can tackle the inverse problem. For the parametrization 
of the interior boundary, we choose $16$ Fourier coefficients in 
space ($N_K=8)$ and $10$ Legendre polynomials in time ($N_L=9$), 
leading to $160$ design parameters in total. As an initial guess for 
the free inner boundary, we choose the cylinder of radius $0.3$. 
We perform 100 iterations in the optimization procedure and use a 
quasi Newton method updated by the limited memory inverse BFGS 
rule, where 10 updates are stored, see \cite{Nocedal_BFGS} for example.

\pgfplotstableread{grad.txt}
\datatablegrad%
\pgfplotstableread{func.txt}
\datatablefunct%
\pgfplotstableread{err.txt}
\datatableerr%
\begin{minipage}[hbt]{\textwidth}
\centering
\begin{tikzpicture}[baseline=(current axis.south)]
  \begin{axis}[width=0.5\textwidth, title={Value of functional}, xlabel={Iteration}]
    \addplot +[mark=none] table[x expr=\coordindex+1, y index=0] from \datatablefunct;
  \end{axis}
\end{tikzpicture}
\quad
\begin{tikzpicture}[baseline=(current axis.south)]
  \begin{axis}[width=0.5\textwidth, title={$\ell^{\infty}$-norm of gradient}, xlabel={Iteration}]
    \addplot +[mark=none] table[x expr=\coordindex+1, y index=0] from \datatablegrad;
  \end{axis}
\end{tikzpicture}
\captionof{figure}{\label{fig.funct_and_grad}}The histories of the 
functional (\emph{left}) and of the shape gradient (\emph{right}). 
\end{minipage}

In Figure~\ref{fig.funct_and_grad} on the left, the evolution of the 
shape gradient during the course of the minimization algorithm is 
shown, while on the right the evolution of the functional is displayed.
In Figure~\ref{fig.l2_error}, we can see the $\ell^2$-error in the shape 
coefficients corresponding to the shape error. We clearly observe 
convergence of the minimization algorithm. 

\begin{minipage}[hbt]{\textwidth}
\centering
\begin{tikzpicture}[baseline=(current axis.south)]
  \begin{axis}[extra y ticks={0.08}, width=0.5\textwidth, title={$\ell^2$-error of the shape}, xlabel={Iteration}]
    \addplot +[mark=none] table[x expr=\coordindex+1, y index=0] from \datatableerr;
  \end{axis}
\end{tikzpicture}

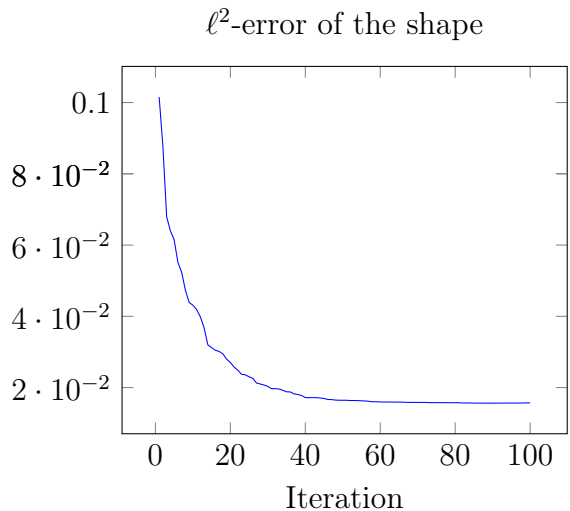
\captionof{figure}{\label{fig.l2_error} $\ell^2$-error 
of the shape coefficients corresponding to the difference 
in the shapes.}
\end{minipage}

In Figure~\ref{fig.shape_output}, we present the final reconstruction 
of the shape, where the wireframe corresponds to the exact shape
and the solid shape is its reconstruction. It can be figured out that 
the final reconstruction of the shape is not very good at the starting 
time $t=0$ and the stopping time $T=1$. But in between, the shape 
is very well reconstructed. 

\begin{figure}[htb]
\centering
\begin{subfigure}[t]{.25\textwidth}
  \centering
  \includegraphics[width=\linewidth, trim={13cm 3cm 13cm 3cm},clip]{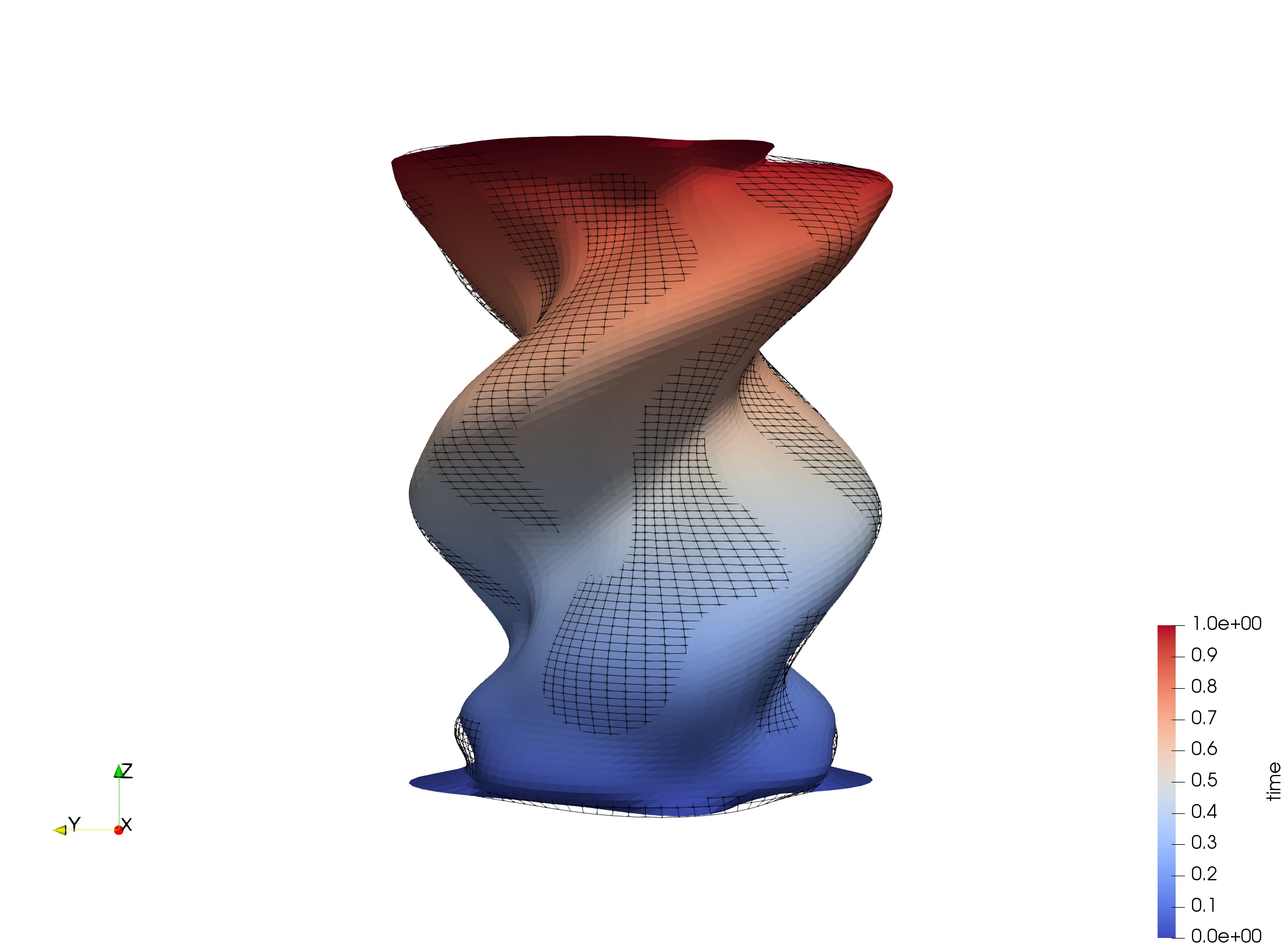}
  \caption{View with the\\ $x$-axis in front.}
  \label{pic.meshgen2_eliminate_flip}
\end{subfigure}\quad
\begin{subfigure}[t]{.25\textwidth}
  \centering
  \includegraphics[width=\linewidth, trim={13cm 3cm 13cm 3cm},clip]{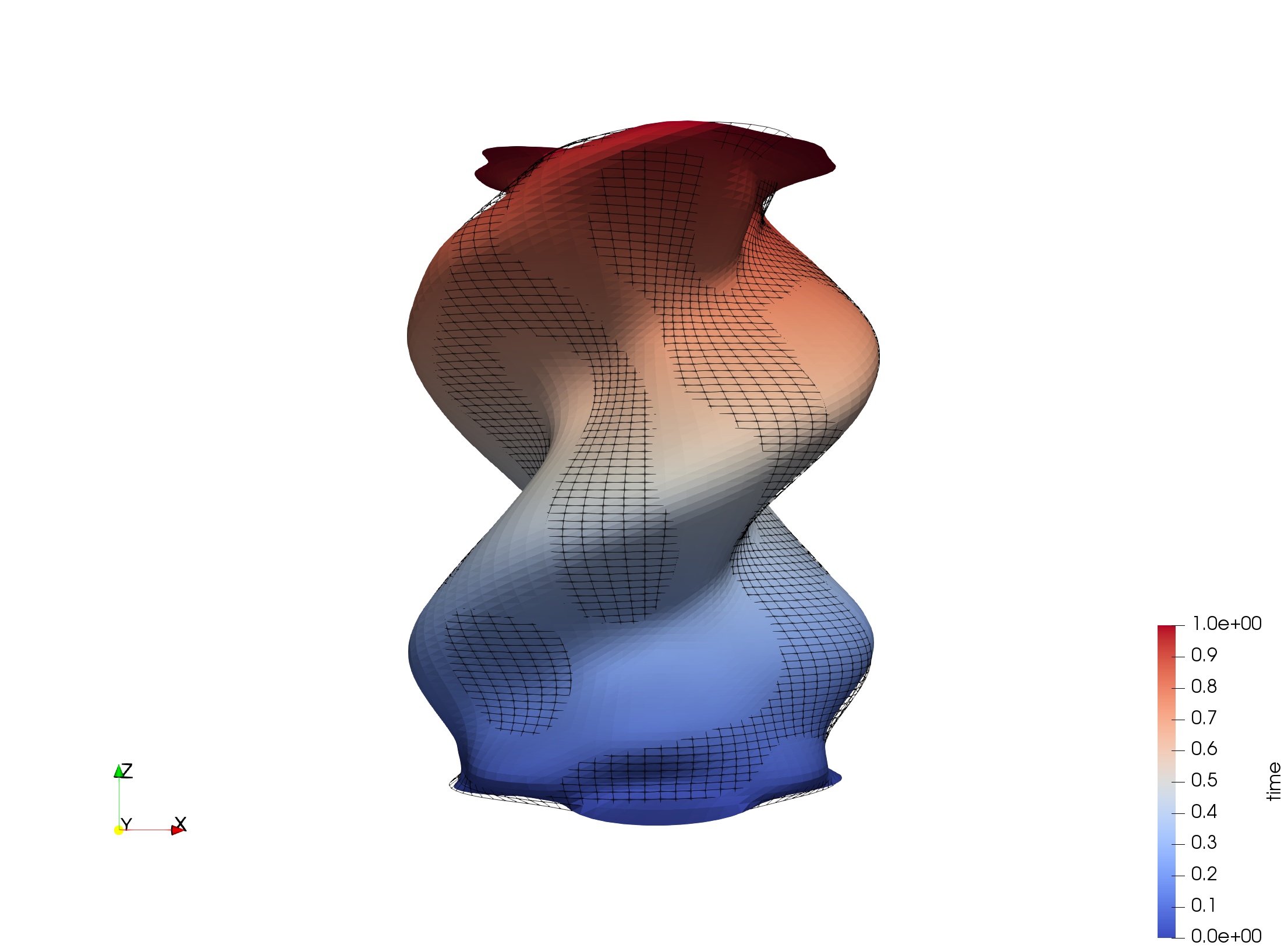}
  \caption{View with the\\ $y$-axis in front.}
  \label{pic.meshgen2_eliminate}
\end{subfigure}\quad
\begin{subfigure}[t]{.25\textwidth}
  \centering
  \includegraphics[width=\linewidth, trim={13cm 3cm 12cm 3cm},clip]{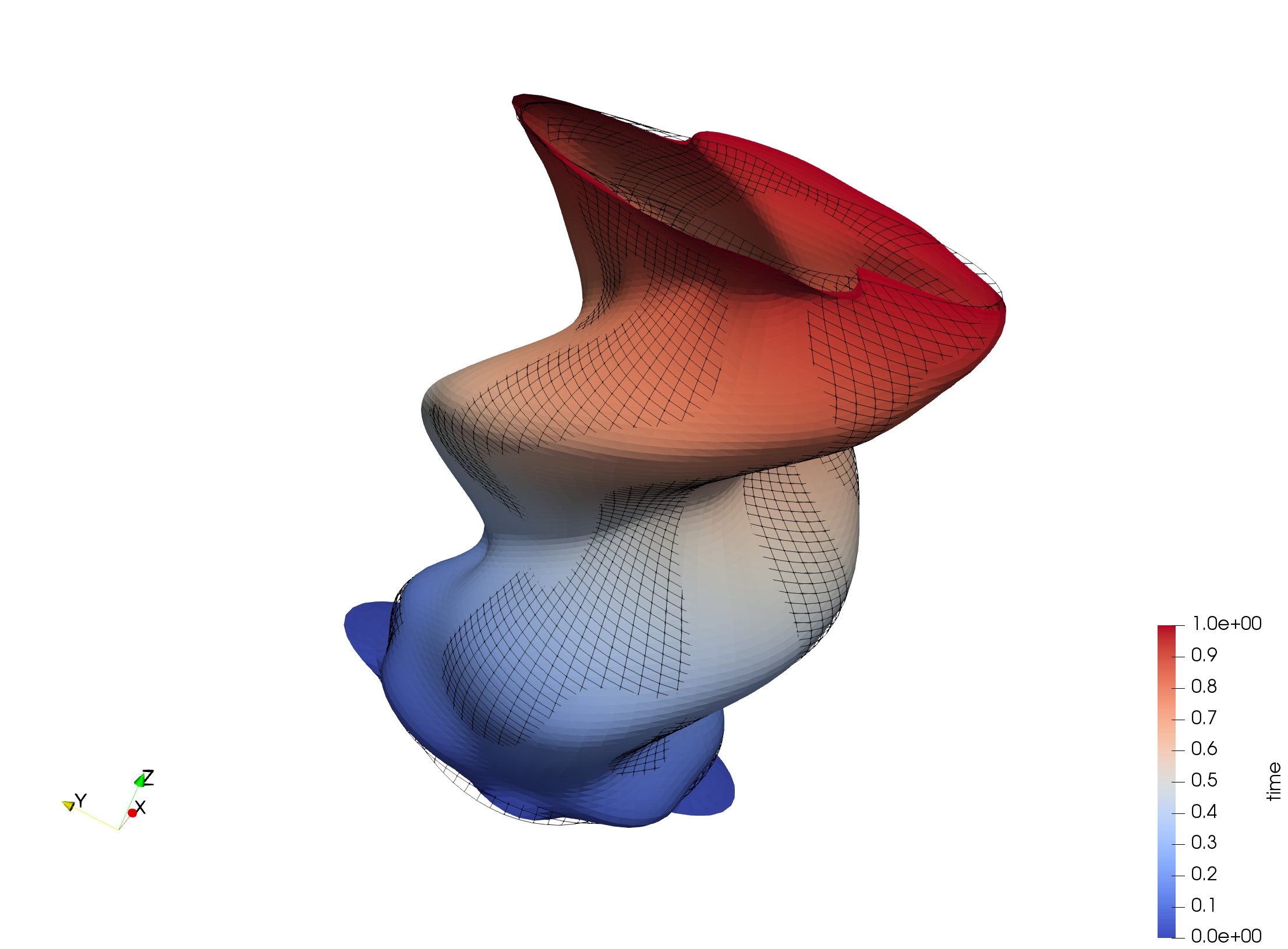}
  \caption{Three-dimensional view. }
  \label{pic.meshgen2_eliminate_flip_eliminate_centroid}
\end{subfigure}
\caption{\label{fig.shape_output} The desired shape as 
a wireframe together with the reconstructed shape in solid. 
The time corresponds to the $z$-axis.}
\end{figure}

\section{Conclusion}\label{sec.conclusion}
In this article, we solved a time-dependent shape 
reconstruction problem by means of shape optimization.
We computed the shape derivative of the tracking-type 
functional for the Neumann data with the help of the 
perturbation of identity. It turned out that this shape derivative 
coincides with the one when the void is time-indepen\-dent. 
We also demonstrated by numerical experiments that it is indeed 
possible to reconstruct a time-dependent shape by the proposed 
approach. By restricting to star-shaped voids, we have been able 
to compute the error between the desired shape and the reconstructed 
shape. The convergence of the minimization algorithm has clearly 
been observed.

\appendix
\section{Local shape derivative}\label{sec.proof_local_shape}
The proof of the local shape derivative follows the lines of 
\cite{Chapko_et_al}. We state here the adjustment to the 
time-dependent setting.

We first present two general lemmas, which are used later. 
We consider a mapping $\xibfm$, which maps a domain 
$\Omega_{\tau}$ to a domain $\Omega_{\varsigma}$ and 
satisfies a uniformity condition as in \eqref{eq.uniformity_cond}. 
We will use the lemmas for $\xibfm = \kabfm$ and $\xibfm = \mathbf{I} + s \Zbfm$. 
Let us denote $Q_{\tau} = \cup_{ \tau} \lbrace \tau \rbrace \times \Omega_{\tau}$ 
and analogously for $Q_{\varsigma}$ and the lateral area by 
$\Sigma_{\tau}$ or $\Sigma_{\varsigma}$, respectively.

\begin{lemma} \label{lem.id_stefan}
For $v$ smooth enough it holds
\begin{equation}\label{eq.id_stefan1}
   (\nabla v) \circ \xibfm
   	= (\on{D} \xibfm)^{-\intercal} \nabla(v \circ \xibfm) 
\end{equation}
and
\begin{equation}\label{eq.id_stefan2}
\begin{aligned}
(\partial_t v)& \circ \xibfm = \partial_t \big(v \circ \xibfm\big) 
	- \big(\on{D} \xibfm\big)^{-\intercal}
		\nabla\big(v \circ \xibfm\big) \cdot \partial_t \xibfm.
\end{aligned}
\end{equation}
\end{lemma}
\begin{proof}
By the chain rule, we can compute
\[ 
  \nabla\big(v \circ \xibfm \big) 
	=  (\on{D}\xibfm )^{\intercal}(\nabla v)\circ\xibfm
\]
from where \eqref{eq.id_stefan1} follows immediately. 
Moreover, the multivariable chain rule yields
\begin{align*}
\partial_t \big(v \circ \xibfm \big) 
	= (\partial_t v) \circ \xibfm
	+ (\nabla v) \circ \xibfm \cdot \partial_t \xibfm,
\end{align*}
since only the spatial component is affected by the composition 
with $\xibfm$. Using \eqref{eq.id_stefan1}, we get \eqref{eq.id_stefan2}.
\end{proof}
Notice that the identities \eqref{eq.id_stefan1} and \eqref{eq.id_stefan2}
are also stated in \cite[pg.~43]{Moubachir} in the setting of the 
speed method. 

\begin{lemma}
\label{lem.transport_eq_back_0}
Let $v\in \hat{H}^{1,\frac12}(Q_T)$ and $\varphi\in \tilde{H}_0^{1,\frac12}(Q_T)$.
Then, the transport of
\begin{equation}
\label{eq.general_weak_form}
S(v,\varphi):= \int_0^T \int_{\Omega_{\varsigma}} \{\nabla v \cdot \nabla \varphi + \partial_t v \varphi\} \, \drm \xbfm \drm t 
= \int_0^T \int_{\Omega_{\varsigma}} h \varphi \, \drm \xbfm \drm t
\end{equation}
from $Q_{\varsigma}$ to $Q_{\tau}$ gives the parabolic problem
\begin{align}
\label{eq.parabolic_with_a}
\int_0^T \int_{\Omega_{\tau}} \partial_t v^{\tau, \varsigma} \varphi^{\tau, \varsigma}\, \drm \xbfm \drm t 
+ \int_0^T a(t; v^{\tau, \varsigma}, \varphi^{\tau, \varsigma}) \, \drm t 
= \int_0^T \int_{\Omega_{\tau}} h^t \varphi^{\tau, \varsigma} \, \drm \xbfm \drm t
\end{align}
with
\begin{align*}
a(t; v^{\tau, \varsigma}, \varphi^{\tau, \varsigma}) &:= \int_{\Omega_{\tau}} \langle (\on{D} \xibfm)^{-\intercal} \nabla v^{\tau, \varsigma} , (\on{D} \xibfm)^{-\intercal} \nabla \varphi^{\tau, \varsigma} \rangle \, \drm \xbfm \\
&\quad -\int_{\Omega_{\tau}} \langle (\on{D} \xibfm)^{-\intercal} \nabla v^{\tau, \varsigma}, \partial_t \xibfm \varphi ^t \rangle \, \drm \xbfm \\
&\quad - \int_{\Omega_{\tau}} \langle (\on{D} \xibfm)^{-\intercal}  \frac{1}{\det (\on{D} \xibfm)} \nabla \big( \det (\on{D} \xibfm)\big) \varphi^{\tau, \varsigma}, (\on{D} \xibfm)^{-\intercal}  \nabla v^{\tau, \varsigma}  \rangle \, \drm \xbfm , \nonumber \\
\end{align*}
where $v^{\tau, \varsigma} = v \circ \xibfm$ and similarly for $\varphi^{\tau, \varsigma}$ and $h^{\tau, \varsigma}$.
\end{lemma}
\begin{proof}
With the aid of Lemma \ref{lem.id_stefan}, the transport of \eqref{eq.general_weak_form} 
from $Q_{\varsigma}$ onto $Q_{\tau}$ gives
\begin{align*}
& \int_0^T \int_{\Omega_{\tau}} \det (\on{D} \xibfm) (\on{D} \xibfm)^{- \intercal} \nabla (v \circ \xibfm) \cdot (\on{D} \xibfm)^{-\intercal} \nabla (\varphi \circ \xibfm) \, \drm \xbfm \drm t \\
&\quad+ \int_0^T \int_{\Omega_{\tau}} \det (\on{D} \xibfm) \Big[\partial_t ( \varphi \circ \xibfm ) v \circ \xibfm - (\on{D} \xibfm) ^{-\intercal} \nabla (v \circ \xibfm)  \cdot \partial_t \xibfm (\varphi \circ \xibfm)  \Big] \, \drm \xbfm \drm t\\
&\quad\quad = \int_0^T \int_{\Omega_{\tau}} \det (\on{D} \xibfm) (h \circ \xibfm) (\varphi \circ \xibfm) \, \drm \xbfm \drm t.
\end{align*}
Using Green's first identity and the zero boundary condition yields
\begin{align*}
&\int_0^T \int_{\Omega_{\tau}} - \on{div} \big( \det (\on{D \xibfm}) (\on{D} \xibfm)^{-1} (\on{D} \xibfm)^{- \intercal} \nabla (v \circ \xibfm) \big) (\varphi \circ \xibfm) \, \drm \xbfm \drm t\\
&\quad+ \int_0^T \int_{\Omega_{\tau}} \det (\on{D} \xibfm) \Big[ \partial_t (v \circ \xibfm) - (\on{D} \xibfm)^{-\intercal} \nabla (v \circ \xibfm) \cdot \partial_t \xibfm \Big] (\varphi \circ \xibfm) \, \drm \xbfm \drm t\\
& \quad\quad = \int_0^T \int_{\Omega_{\tau}} \det (\on{D} \xibfm) (h \circ \xibfm) (\varphi \circ \xibfm) \, \drm \xbfm \drm t.
\end{align*}
Thus, in the strong formulation, we have when dividing by $\det(\on{D} \xibfm)$ that
\begin{equation*}
\begin{aligned}
& -\frac{1}{\det (\on{D} \xibfm)} \on{div} \Big( \det (\on{D} \xibfm) (\on{D} \xibfm)^{-1} (\on{D} \xibfm)^{- \intercal} \nabla (v \circ \xibfm) \Big) \\
&\hspace*{2cm}  + \partial_t (v \circ \xibfm) - (\on{D} \xibfm)^{-\intercal} \nabla (v \circ \xibfm) \cdot \partial_t \xibfm = h \circ \xibfm \ \ \text{in } Q_{\tau}.
\end{aligned}
\end{equation*}
Rewriting gives
\begin{align*}
& -\on{div} \Big( (\on{D} \xibfm)^{-1} (\on{D} \xibfm)^{-\intercal} \nabla v^{\tau, \varsigma} \Big) +  \partial_t v^{\tau, \varsigma} - (\on{D} \xibfm)^{-\intercal} \nabla v^{\tau, \varsigma} \cdot \partial_t \xibfm\\
&\hspace*{2cm} - \frac{1}{\det (\on{D} \xibfm)} \nabla \big(\det (\on{D} \xibfm) \big) \cdot (\on{D} \xibfm)^{-1} (\on{D} \xibfm)^{- \intercal} \nabla v^{\tau, \varsigma} = h^{\tau, \varsigma}\ \ \text{in } Q_{\tau}.
\end{align*}
Testing again with a function $\varphi^{\tau, \varsigma}$ gives the weak formulation
\begin{align*}
&\int_0^T \int_{\Omega_{\tau}} - \on{div} \Big( (\on{D} \xibfm)^{-1} (\on{D} \xibfm)^{-\intercal} \nabla v^{\tau, \varsigma} \Big) \varphi^{\tau, \varsigma} \, \drm \xbfm \drm t + \int_0^T \int_{\Omega_{\tau}} \partial_t v^{\tau, \varsigma} \varphi^{\tau, \varsigma} \, \drm \xbfm \drm t \\
&\quad- \int_0^T \int_{\Omega_{\tau}} (\on{D} \xibfm)^{-\intercal} \nabla v^{\tau, \varsigma} \cdot \partial_t \xibfm \varphi^{\tau, \varsigma} \, \drm \xbfm \drm t \\
&\quad- \int_0^T \int_{\Omega_{\tau}} \frac{1}{\det (\on{D} \xibfm)} \nabla \big( \det (\on{D} \xibfm)\big) \cdot (\on{D} \xibfm)^{-1} (\on{D} \xibfm)^{-\intercal} \nabla v^{\tau, \varsigma} \varphi^{\tau, \varsigma} \, \drm \xbfm \drm t\\ 
&\quad\quad= \int_0^T \int_{\Omega_{\tau}} h^{\tau, \varsigma}  \varphi^{\tau, \varsigma} \, \drm \xbfm \drm t,
\end{align*}
which can be reformulated by using the divergence theorem with vanishing boundary terms to
\begin{equation}\label{eq.weak_formulation_transported_back}
\begin{aligned}
&\int_0^T \int_{\Omega_{\tau}} (\on{D} \xibfm)^{-\intercal} \nabla v^{\tau, \varsigma} \cdot (\on{D} \xibfm)^{-\intercal} \nabla \varphi^{\tau, \varsigma} \, \drm \xbfm \drm t+ \int_0^T \int_{\Omega_{\tau}} \partial_t v^{\tau, \varsigma} \varphi^{\tau, \varsigma} \, \drm \xbfm \drm t\\
&\quad 
-\int_0^T \int_{\Omega_{\tau}} \langle (\on{D} \xibfm)^{-\intercal} \nabla v^{\tau, \varsigma}, \partial_t \xibfm \varphi ^t \rangle \, \drm \xbfm \drm t \\
&\quad- \int_0^T \int_{\Omega_{\tau}} \langle (\on{D} \xibfm)^{-\intercal}  \frac{1}{\det (\on{D} \xibfm)} \nabla \big( \det (\on{D} \xibfm)\big) \varphi^{\tau, \varsigma}, (\on{D} \xibfm)^{-\intercal}  \nabla v^{\tau, \varsigma}  \rangle \, \drm \xbfm \drm t\\
&\quad\quad= \int_0^T \int_{\Omega_{\tau}} h^{\tau, \varsigma} \varphi^{\tau, \varsigma}  \, \drm \xbfm \drm t. 
\end{aligned}
\end{equation}
From here, the claim follows immediately.
\end{proof}

In order to compute the local shape derivative, we first introduce the 
material derivative to \eqref{eq.cat_diff_eq} as stated in the 
following lemma.

\begin{lemma} \label{lem.material_derivative}
The material derivative of \eqref{eq.cat_diff_eq}, 
which is defined as the limit
\[ 
  \dot{v} : = \lim_{s \to 0} \frac{v^{t,s} - v}{s},
\]
exists in $\hat{H}^{1, \frac{1}{2}}_0 (Q_T)$ and satisfies
\begin{equation}\label{eq.cat_weak_formulation_pure_dirichlet_material}
  S(\dot{v},\varphi) = G(\varphi)\ \text{for all}\ \varphi \in \tilde{H}_0^{1, \frac{1}{2}} (Q_T),  
\end{equation}
where $S$ is given by \eqref{eq.weak_form_pde} and
\begin{equation} \label{eq.rhs_material_der}
  G(\varphi ) = \int_0^T \int_{\Omega_t} \big\{(\on{D}\Zbfm + \on{D}\Zbfm^{\intercal}) 
  	\nabla v \cdot \nabla \varphi + \varphi \nabla(\on{div}\Zbfm) \cdot \nabla v 
		+ (\partial_t \Zbfm) \cdot \nabla v \varphi\big\}\, \drm \xbfm \drm t. 
\end{equation}
\end{lemma}

\begin{proof}
Let $v_{t,s}$ be the solution of the perturbed problem on $Q_T^s$,
satisfying the same boundary conditions as stated in \eqref{eq.cat_diff_eq}. 
As an immediate consequence of \cite[Chapter IV, Theorem 9.1]{Lady}, 
the solution $v_{t,s}$ lies in $\hat{H}^{2,1}(Q_T^s)$ under our smoothness 
assumptions. Notice that the increased regularity of the solution of the differential 
equation is needed for the boundary condition of the local shape 
derivative \eqref{eq.local_shape_der}. 

We have for the perturbed bilinear form
\begin{equation}\label{eq:perturbedBLF}
S_s(v_{t,s}, \varphi) := \int_0^T \int_{\Omega_{t,s}} \{\partial_t v_{t,s} \varphi
	+ \nabla v_{t,s} \cdot \nabla \varphi \} \, \drm \xbfm \drm t, 
\end{equation}
that $S_s(v_{t,s}, \varphi ) = 0$ for all $\varphi \in \tilde{H}_0^{1, \frac{1}{2}}
(Q_T^s)$. The existence and uniqueness of a solution follows as in Theorem 
\ref{thm.solution_operator_isomorphism} by using that the transformation 
$\kabfm + s \Zbfm \circ \kabfm$ satisfies again a uniformity condition as 
stated in \eqref{eq.uniformity_cond}. With similar computations as in the 
proof of Lemma \ref{lem.transport_eq_back_0}, when setting 
$\xibfm = \mathbf{I} + s \Zbfm$, $\Omega_{\tau} = \Omega_t$ 
and $\Omega_{\varsigma} = \Omega_{t,s}$,  the transformation 
of the integral in \eqref{eq:perturbedBLF} back onto $\Omega_t$ reads
%
\begin{align*}
 S_s(v_{t,s}, \varphi) &= \int_0^T \int_{\Omega_t} \det\big(\on{D}(\mathbf{I} + s \Zbfm)\big)   
  	\Big[\big\{\partial_t v^{t,s} - \big(\on{D}(\mathbf{I} + s \Zbfm)\big)^{-\intercal} 
		\nabla v^{t,s}  \cdot \partial_t (\mathbf{I} + s \Zbfm)\big\}\varphi^s\\
  &\hspace*{3cm} + \big(\on{D}(\mathbf{I} + s \Zbfm)\big)^{-\intercal} 
  	\nabla v^{t,s} \cdot \big(\on{D}(\mathbf{I} + s \Zbfm)\big)^{-\intercal} 
		\nabla \varphi^s\Big]\, \drm \xbfm \drm t ,
\end{align*}
where we have set $v^{t,s} := v_{t,s} \circ (\mathbf{I} + s \Zbfm)$ 
and $\varphi^s$ analogously. We define this bilinear form on the unperturbed domain as
\begin{equation*}
\begin{aligned}
 S^s(w,\varphi) & := \int_0^T \int_{\Omega_t}\det\big(\on{D}(\mathbf{I} + s \Zbfm)\big) \\
 & \left[{\bf B}^s \nabla w \cdot \nabla \varphi + \partial_t w \varphi 
 	- \big(\on{D}(\mathbf{I} + s \Zbfm)\big)^{-\intercal} \nabla w  
		\cdot \partial_t  (\mathbf{I} + s \Zbfm) \varphi \right] \, \drm \xbfm \drm t,
\end{aligned}
\end{equation*}
where 
\[
  {\bf B}^s = \big(\on{D}(\mathbf{I} + s \Zbfm)\big)^{-1}
  		\big(\on{D}(\mathbf{I} + s \Zbfm)\big)^{-\intercal}. 
\]  
Note that the last term in the definition of $S^s(w,\varphi)$ is new 
in comparison with \cite{Chapko_et_al}. 

We conclude the following statement: 
\[ 
S_s(v_{t,s}, \varphi) =0\ \text{for all}\ 
	\varphi \in \tilde{H}_0^{1,\frac{1}{2}} (Q_T^s) 
\]
for $v_{t,s}\in\hat{H}^{2,1}(Q_T^s)$ is equivalent to
\begin{equation}\label{eq.cat_local_dirichlet_transformed_back}
 S^s(v^{t,s},\varphi) = 0\ \text{for all}\ 
 	\varphi \in \tilde{H}_0^{1, \frac{1}{2}}(Q_T)
\end{equation}
for $v^{t,s}\in\hat{H}^{2,1} (Q_T)$. Integration by parts, 
where we use the zero boundary values of the test function, 
and dividing by $\det\big(\on{D}(\mathbf{I} + s \Zbfm)\big)$
verifies that \eqref{eq.cat_local_dirichlet_transformed_back} 
is equivalent to the formulation 
\begin{equation}\label{eq.cat_strong_local_dirchlet}
\begin{aligned}
&\partial_t v^{t,s} - \big(\on{D}(\mathbf{I} + s \Zbfm)\big)^{-\intercal}
	\nabla v^{t,s} \cdot \partial_t (\mathbf{I} + s \Zbfm) \\
&\qquad- \frac{1}{\det(\on{D}(\mathbf{I} + s \Zbfm))} \nabla\Big(\det\big(\on{D}(\mathbf{I} + s \Zbfm)\big)\Big) 
	\cdot {\bf B}^s \nabla v^{t,s}- \on{div}({\bf B}^s \nabla v^{t,s}) = 0 \\
&\hspace*{8cm}\text{in}\ \bigcup_{0<t<T} (\{t\}\times\Omega_t).
\end{aligned}
\end{equation}
Because of $S(v, \varphi) = 0 $ and $S^s(v^{t,s}, \varphi) =0$, it holds
\[
S (v^{t,s}-v, \varphi) = - S^s(v^{t,s}, \varphi ) + S(v^{t,s}, \varphi). 
\]
We can therefore consider
\[ \frac{1}{s} S (v^{t,s}-v, \varphi) = G_s(\varphi)
	\ \text{for all}\ \varphi \in \tilde{H}_0^{1, \frac{1}{2}} (Q_T)
\]
for the computation of the material derivative, where
\begin{align*}
G_s(\varphi) & = \frac{1}{s} \int_0^T \int_{\Omega_t} \Big\{
	-\det\big(\on{D}(\mathbf{I} + s \Zbfm)\big){\bf B}^s \nabla v^{t,s} \cdot \nabla \varphi 
		- \det\big(\on{D}(\mathbf{I} + s \Zbfm)\big)\partial_t v^{t,s} \varphi  \\
& + \det\big(\on{D}(\mathbf{I} + s \Zbfm)\big)\big(\on{D}(\mathbf{I} + s \Zbfm)\big)^{-\intercal} 
	\nabla v^{t,s} \cdot \partial_t (\mathbf{I} + s \Zbfm) \varphi  \\
& + \partial_t v^{t,s} \varphi + \nabla v^{t,s} \cdot \nabla \varphi\Big\}\, \drm \xbfm \drm t.
\end{align*}
Herein, the second line is new in comparison with \cite{Chapko_et_al}. 

We reformulate the expression for $G_s(\varphi)$ 
the same way as in \cite{Chapko_et_al} and we arrive at
\begin{align*}
G_s(\varphi) &= \frac{1}{s}\int_0^T \int_{\Omega_t}\bigg\{[{\bf I} - {\bf B}^s] \nabla v^{t,s} \cdot \nabla \varphi \\
	&\hspace*{2cm}+ \frac{\varphi}{\det\big(\on{D}(\mathbf{I} + s\Zbfm)\big)}
		\nabla\Big(\det\big(\on{D}(\mathbf{I} + s\Zbfm)\big)\Big) 
		\cdot {\bf B}^s \nabla v^{t,s}\bigg\}\,\drm \xbfm \drm t \\
& + \frac{1}{s}\int_0^T \int_{\Omega_t} \Big\{\det\big(\on{D}(\mathbf{I} + s\Zbfm)\big)
	\big(\nabla v^{t,s}\big)^{\intercal} 
		\big(\on{D}(\mathbf{I} + s \Zbfm)\big)^{-1} \partial_t (\mathbf{I} + s \Zbfm) \varphi\Big\}\, \drm \xbfm \drm t,
\end{align*}
where the last line is new in this time-dependent setting 
in comparison with the proof given in \cite{Chapko_et_al}.
We now need to show that $G_s$ converges to $G$ stated in \eqref{eq.rhs_material_der}.

Clearly, $\varphi \mapsto G_s(\varphi)$ is a bounded 
linear functional on $\tilde{H}_0 ^{1, \frac{1}{2}} (Q_T)$, i.e.\ 
$G_s\in\Big(\tilde{H}_0^{1,\frac{1}{2}}(Q_T)\Big)'$. 
Therefore, we can interchange the integration and the limes
$s\to 0$. Especially, as in \cite{Chapko_et_al}, we have 
\[ 
  \frac{1}{s}({\bf I}-{\bf B}^s)\to\on{D}\Zbfm+\on{D}\Zbfm^{\intercal} 
\]
and
\[ 
  \frac{1}{s\det\big(\on{D}(\mathbf{I} + s\Zbfm)\big)}
		\nabla\Big(\det\big(\on{D}(\mathbf{I} + s\Zbfm)\big)\Big)  \to \nabla \on{div} \Zbfm 
\]
as $s \to 0$. Thus, it remains to compute
\[
  \lim_{s \to 0} \frac{1}{s}\big(\on{D}(\mathbf{I} + s \Zbfm)\big)^{-1}\partial_t(\mathbf{I}+s \Zbfm). 
\]
By using the Neumann series, we have
\[
\big(\on{D}(\mathbf{I}+s \Zbfm)\big)^{-1} = \mathbf{I} - s\on{D}\Zbfm + o(s) 
\]
and therefore 
\[ 
  \lim_{s \to 0} \frac{1}{s}\big(\on{D}(\mathbf{I} + s \Zbfm)\big)^{-1}\partial_t(\mathbf{I}+s \Zbfm)
  	= \lim_{s \to 0}\frac{1}{s}\big(\mathbf{I}-s\on{D}\Zbfm + o(s)\big)s\partial_t \Zbfm  =  \partial_t \Zbfm.
\]

In order to conclude the convergence $G_s\to G$ as $s\to 0$, we need that 
$v^{t,s}$ converges to $v$ in $H^{1,0}(Q_T)$. To this end, we transform the 
equations for $v$ and for $v^{t,s}$ to $Q_0$ by using the transformation 
$\kabfm$, yielding two differential equations similar to 
\eqref{eq.cat_strong_local_dirchlet}. Applying \cite[Theorem 4.5 on pg.~166]{Lady} 
implies the convergence of $v_{t,s} \circ (\mathbf{I} + s \Zbfm) \circ \kabfm$ to 
$v  \circ \kabfm$ and thus, with the uniformity condition \eqref{eq.uniformity_cond},
also $v^{t,s}$ converges to $v$. Therefore, we have convergence of $G_s \to G$ 
as $s \to 0$ in the dual space of $\tilde{H}^{1,\frac{1}{2}}_0 (Q_T)$ as in 
\cite{Chapko_et_al}, with $G(\varphi)$ as in \eqref{eq.rhs_material_der}.

Now, we can argue as in \cite{Chapko_et_al}: Since the solution 
operator is an isomorphism from $\hat{H}^{-1,-\frac{1}{2}}(Q_T)$ 
to $\hat{H}^{1,\frac{1}{2}}_0 (Q_T)$ (see Theorem 
\ref{thm.solution_operator_isomorphism}), the statement in 
Lemma \ref{lem.material_derivative} is true.
\end{proof}

Having the material derivative for \eqref{eq.cat_diff_eq}
at hand, we are finally in the position to prove the local
shape derivative posed in Theorem \ref{thm.local_shape_der}.

\begin{proof}[Proof of Theorem \ref{thm.local_shape_der}] 
Starting from the material derivative, we would like to 
compute the local shape derivative $\delta v$. 

If we consider $v \in \hat{H}^{2,1}(Q_T)$, we also have $\nabla v \in H^{1, \frac{1}{2}} (Q_T)$ and 
$\Delta v \in L^2(Q_T)$, as in \cite{Chapko_et_al}. This follows from $\kabfm$ 
being a diffeomorphism and from the time-independent case in \cite[Proposition 2.3 on pg.~14 with $r=2$, 
$s=1$, $j=2$ and $k = 0$]{lions2}. Then for $\varphi \in V(Q_T)$ (see \eqref{eq.space_V}
for the definition of the space), we have the same identity as in 
\cite[pg.~859]{Chapko_et_al}, namely
\begin{align*}
(\on{D}\Zbfm + \on{D}\Zbfm^{\intercal}) \nabla v \cdot \nabla \varphi 
	+ \varphi \nabla(\on{div} \Zbfm) \cdot \nabla v 
  &= \on{div} \big( \on{div}(\varphi \Zbfm) \nabla v 
  	- (\nabla v \cdot \nabla \varphi) \Zbfm \big) \\
  & \qquad + \nabla(\Zbfm \cdot \nabla v) \cdot \nabla \varphi 
  	- \on{div}(\varphi \Zbfm) \Delta v.
\end{align*}
Applying this identity and the divergence theorem to \eqref{eq.rhs_material_der} yields
\[
G(\varphi) = \int_0^T \int_{\Omega_t}\big\{\nabla (\Zbfm \cdot \nabla v) \cdot \nabla \varphi 
  - \on{div} (\varphi \Zbfm) \underbrace{\Delta v}_{= \partial_t v}
  + \nabla v \cdot \partial_t \Zbfm \varphi\big\}\, \drm \xbfm \drm t,
\]
where the boundary terms vanish due to the compact support of $\varphi$.
Note 
that only the last term of the integrand differs from the computations 
in \cite{Chapko_et_al}. It holds
\[ - \partial_t v \on{div}(\Zbfm \varphi) = - \on{div} (\partial_t v \Zbfm \varphi) 
	+ \Zbfm \varphi \cdot \nabla(\partial_t v) \]
and, therefore, we can apply the divergence theorem 
again to get
\[
G(\varphi) =  \int_0^T \int_{\Omega_t } \big\{\nabla (\Zbfm \cdot \nabla v) \cdot \nabla \varphi  
	+ \Zbfm \varphi \nabla (\partial_t v)
	+ \nabla v \cdot \partial_t \Zbfm \varphi\big\}\, \drm \xbfm \drm t.
\]
Taking the two time derivatives together yields
\[
  G(\varphi) =  \int_0^T \int_{\Omega_t }
	\big\{\partial_t (\nabla v \cdot \Zbfm) \varphi
		+ \nabla (\Zbfm \cdot \nabla v) \cdot \nabla \varphi\big\}\drm \xbfm \drm t.
\]
This is the same expression as in \cite{Chapko_et_al}. Thus, the 
local shape derivative satisfies the same partial differential equation as 
in \cite{Chapko_et_al} except for being in a space-time tube $Q_T$ 
instead a space-time cylinder $Q_0$.
\end{proof}

\bibliographystyle{plain} 
\bibliography{literatur}

\begin{thebibliography}{10}

\bibitem{Chapko_et_al}
R.~Chapko, R.~Kress, and J.-R. Yoon.
\newblock On the numerical solution of an inverse boundary value problem for
  the heat equation.
\newblock {\em Inverse Problems}, 14(4):853--867, 1998.

\bibitem{chapko_neumann}
R.~Chapko, R.~Kress, and J.-R. Yoon.
\newblock An inverse boundary value problem for the heat equation: the neumann
  condition.
\newblock {\em Inverse problems}, 15(4):1033, 1999.

\bibitem{Costabel}
M.~Costabel.
\newblock Boundary integral operators for the heat equation.
\newblock {\em Integral Equations and Operator Theory}, 13(4):498--552, 1990.

\bibitem{Delfour_Zolesio}
M.C. {D}elfour and J.-P. {Zol\'{e}sio}.
\newblock {\em Shapes and Geometries: Metric, Analysis, Differential Calculus,
  and Optimization}.
\newblock Advances in Design and Control, SIAM, USA, second edition, 2011.

\bibitem{Dennis_Schnabel}
J.E. {Dennis} and R.B. {Schnabel}.
\newblock {\em Numerical Methods for Nonlinear Equations and Unconstrained
  Optimization Techniques}.
\newblock Prentice-Hall, Englewood Cliffs, 1983.

\bibitem{Dziri_dynamical}
R.~Dziri and J.-P. Zol\'{e}sio.
\newblock Dynamical shape control in non-cylindrical {N}avier-{S}tokes
  equations.
\newblock {\em Journal of Convex Analysis}, 6(2):293--318, 1999.

\bibitem{Dziri}
R.~Dziri and J.-P. Zol\'{e}sio.
\newblock Eulerian derivative for non-cylindrical functionals.
\newblock In M.~P.~{Polis} J.~{Cagol} and J.-P. {Zol\'{e}sio}, editors, {\em
  Shape optimization and optimal design}, pages 87--107. Lecture notes in pure
  and applied mathematics, Marcel Dekker, Inc., New York, Basel, 2001.

\bibitem{Fletcher}
R.~{Fletcher}.
\newblock {\em Practical Methods for Optimization}.
\newblock Wiley, New York, 1980.

\bibitem{Geiger_Kanzow_unres}
C.~{Geiger} and C.~{Kanzow}.
\newblock {\em Numerische Verfahren zur Lösung unrestringierter
  Optimierungsaufgaben}.
\newblock Springer, Berlin-Heidelberg, 1999.

\bibitem{Gurtin}
M.E. {Gurtin}.
\newblock {\em An Introduction to Continuum Mechanics}.
\newblock Academic Press, INC, New York, 1981.

\bibitem{harbrecht_domain_map}
H.~Harbrecht, M.~Peters, and M.~Siebenmorgen.
\newblock Analysis of the domain mapping method for elliptic diffusion problems
  on random domains.
\newblock {\em Numerische Mathematik}, 134(4):823--856, 2016.

\bibitem{HHTauschCat}
H.~Harbrecht and J.~Tausch.
\newblock On the numerical solution of a shape optimization problem for the
  heat equation.
\newblock {\em SIAM J. Sci. Comput.}, 35(1):A104--A121, 2013.

\bibitem{Lady}
O.A. Ladyzenskaja, V.A. Solonnikov, and N.N. Ural'Ceva.
\newblock {\em Linear and Quasilinear Equations of Parabolic Type (Providence,
  RI: American Mathematical Society)}.
\newblock American Mathematical Society, Rhode Island, 1968.

\bibitem{lions2}
J.L. Lions and E.~Magenes.
\newblock {\em Probl{\`e}mes aux limites non homog{\`e}nes et applications},
  volume~2 of {\em Travaux et recherches math{\'e}matiques}.
\newblock Dunod, Paris, 1968.

\bibitem{lions_magenes_v1}
J.L. Lions and E.~Magenes.
\newblock {\em Non-Homogeneous Boundary Value Problems and Applications I}.
\newblock Springer-Verlag, Berlin, G{\"o}ttingen, Heidelberg, 1972.

\bibitem{lions_magenes_v2}
J.L. Lions and E.~Magenes.
\newblock {\em Non-Homogeneous Boundary Value Problems and Applications II}.
\newblock Springer-Verlag, Berlin, G{\"o}ttingen, Heidelberg, 1972.

\bibitem{McLean}
W.~{McLean}.
\newblock {\em Strongly elliptic systems and boundary integral equations}.
\newblock Cambridge University Press, Cambridge, 2000.

\bibitem{Moubachir}
M.~{Moubachir} and J.-P. {Zol\'{e}sio}.
\newblock {\em Moving Shape Analysis and Control}.
\newblock Chapman \& Hall /CRC, Tayler \& Francis Group, USA, 2006.

\bibitem{Nocedal_BFGS}
J.~Nocedal and S.T. Wright.
\newblock {\em Numerical Optimization}.
\newblock Springer Science+Business Media, LLC, second edition, 2006.

\bibitem{numerical_recipes}
W.H. Press, S.A. Teukolsky, W.T. Vetterling, and B.P. Flannery.
\newblock Numerical recipes in fortran 77, vol. 1.
\newblock {\em New York, NY: Press Syndicate of the University of Cambridge,
  Cambridge}, 1992.

\bibitem{sokolowski1988shape}
J.~Sokolowski.
\newblock Shape sensitivity analysis of boundary optimal control problems for
  parabolic systems.
\newblock {\em SIAM journal on control and optimization}, 26(4):763--787, 1988.

\bibitem{Soko_Zolesio}
J.~{Sokolowski} and J.-P. {Zol\'{e}sio}.
\newblock {\em Introduction to Shape Optimization}.
\newblock Springer, Berlin-Heidelberg, 1992.

\bibitem{spurk}
J.H. Spurk and N.~Aksel.
\newblock {\em Fluid Mechanics}.
\newblock Springer-Verlag, Berlin Heidelberg, 2 edition, 2008.

\bibitem{tausch09a}
J.~Tausch.
\newblock Nystrom discretization of parabolic boundary integral equations.
\newblock {\em Appl. Numer. Math.}, 59(11):2843--2856, 2009.

\bibitem{tausch18}
J.~Tausch.
\newblock Nyström method for {BEM} of the heat equation with moving
  boundaries.
\newblock Tech Report, Southern Methodist University, 2018.

\bibitem{El_Yacoubi}
S.~El Yacoubi and J.~Sokolowski.
\newblock Domain optimization problems for parabolic control systems.
\newblock {\em Applied Mathematics and Computer Science}, 6:277--290, 1996.

\bibitem{zolesio_identification}
J.-P. Zol\'{e}sio.
\newblock {\em Identification de domaines par d{\'e}formations}.
\newblock PhD thesis, Universit\'{e} de Nice, 1979.

\end{thebibliography}
\end{document}